\documentclass[a4paper,10pt]{amsart}
\usepackage{graphicx}
\usepackage{psfrag}
\usepackage{amsmath,amssymb}
\usepackage{verbatim}
\usepackage[colorlinks=true]{hyperref}

\usepackage{tikz}
\usetikzlibrary{snakes,trees}

\newtheorem{thm}{Theorem}[section]
\newtheorem{prop}[thm]{Proposition}
\newtheorem{lemma}[thm]{Lemma}
\newtheorem{cor}[thm]{Corollary}
\newtheorem{conj}[thm]{Conjecture}
\newtheorem{fact}[thm]{Fact}
\theoremstyle{definition}
\newtheorem{defi}[thm]{Definition}

\newcommand{\ma}{m^{(A)}}
\newcommand{\mb}{m^{(B)}}
\newcommand{\si}{\sigma}
\newcommand{\C}{C^{-1}}
\newcommand{\lb}{[\![}
\newcommand{\rb}{]\!]}

\title[Polynomials counting FPL configurations]{On some polynomials enumerating Fully Packed Loop configurations}
\author{Tiago Fonseca}
\address{LPTHE (CNRS, UMR 7589), Universit\'e Pierre et Marie Curie-- PARIS 6, 75252 PARIS CEDEX}
\email{fonseca@lpthe.jussieu.fr}
\author{Philippe Nadeau}
\address{Fakult\"at f\"ur Mathematik, Universit\"at Wien, Nordbergstra{\ss}e 15, A-1090 WIEN, AUSTRIA. }
\email{philippe.nadeau@univie.ac.at}
\begin{document}

\tikzstyle{arche} = [red, thick]
\tikzstyle{line} = [black, semithick]
\tikzstyle{dyck} = [black, thick]
\tikzstyle{young} = [black, semithick]

\begin{abstract}
We are interested in the enumeration of Fully Packed Loop configurations on a grid with a given noncrossing matching. By the recently proved Razumov--Stroganov conjecture, these quantities also appear as groundstate components in the Completely Packed Loop model.

 When considering matchings with $p$ nested arches, these numbers are known to be polynomials in $p$. In this article, we present several conjectures about these polynomials: in particular, we describe all real roots, certain values of these polynomials, and conjecture that the coefficients are positive. The conjectures, which are of a combinatorial nature, are supported by strong numerical evidence and the proofs of several special cases. We also give a version of the conjectures when an extra parameter $\tau$ is added to the equations defining the groundstate of the Completely Packed Loop model.
\end{abstract}

\maketitle


 \section*{Introduction}

The recently proved Razumov--Stroganov conjecture~\cite{RS-conj,ProofRS} is a correspondence between, on the one hand, combinatorially defined quantities called Fully Packed Loop (FPL) configurations, and on the other hand, components of the groundstate vector of the Hamiltonian in the Completely Packed Loop model. These quantities are indexed by noncrossing, perfect matchings $\pi$ of $2n$ points (cf. definition in Section~\ref{representations}).The number of FPL configurations with associated matching $\pi$ will be denoted $A_\pi$, while the corresponding components of the groundstate vector in the Completely Packed Loop model are written $\Psi_\pi$. The  Razumov--Stroganov conjecture states then that $A_\pi=\Psi_\pi$ for any $\pi$. 

The goal of this article is to exhibit some surprising properties of these numbers when one studies matchings with nested arches $(\pi)_p=(\cdots(\pi)\cdots)$, which means that there are $p$ nested arches above the matching $\pi$. It was conjectured in ~\cite{Zuber-conj}, and subsequently proved in~\cite{CKLN,artic47}, that the quantities $A_{(\pi)_p}$ and $ \Psi_{(\pi)_p}$ are polynomial in $p$. We define then the polynomial $A_\pi(t)$ such that $A_\pi(p)=A_{(\pi)_p}$ when $p$ is a nonnegative integer.
\medskip

This paper deals with certain conjectures about these polynomials. Let $\pi$ be a matching with $n$ arches: the main conjectures deal with the description of real roots of the polynomials (Conjecture~\ref{conj:realroots}), their values at negative integers between $1-n$ and $-1$ (Conjecture~\ref{conj:dec}), evaluations at $-n$ (Conjecture~\ref{conj:gpi})  and finally the positivity of the coefficients (Conjecture~\ref{conj:posX}). 
We gather some evidence for the conjectures, and prove some special cases (cf. Theorem ~\ref{th:subleading} and Theorem~\ref{th:firstroot}). In the Completely Packed Loop model, one can in fact define bivariate polynomials $\Psi(\tau,t)$ that coincide with $\Psi(t)$ at $\tau=1$; it turns out that most of our conjectures admit a natural generalization in this context also, which in some sense is more evidence for the original conjectures.

 We believe these conjectures can help us understand better the numbers $A_\pi$. Moreover, our work on these conjectures has some interesting byproducts: first, the conjectured root multiplicities of the polynomials $A_\pi(t)$ have nice combinatorial descriptions in terms of $\pi$ (see Section~\ref{sub:combdef}). Then, from the proof of Theorem ~\ref{th:subleading}, we deduce some nice formulas about products of hook lengths of partitions (Proposition~\ref{prop:newhookformulas}). Also, the proof of Theorem~\ref{th:firstroot} involves the introduction of a new multivariate integral.
 \medskip

Let us give a detailed outline of this article, where $\pi$ will refer to a matching with $n$ arches. In Section~\ref{sec:defi}, we define the quantities $A_\pi$ and $\Psi_\pi$, and formulate the Razumov--Stroganov conjecture. We introduce in Section~\ref{sec:polynomials} the central objects of our study, the polynomials $A_\pi(t)$. It is also recalled how to approach the computation of these polynomials.

The main conjectures about the $A_\pi(t)$ are gathered in Section~\ref{sec:conj}: they are Conjectures~\ref{conj:realroots}, ~\ref{conj:dec},~\ref{conj:gpi} and ~\ref{conj:posX}. We give also numerous evidence for these conjectures, the most important one being perhaps that they have been checked for all matchings with $n\leq 8$.

The next two sections address particular cases of some of the conjectures: in Section~\ref{sec:subleading}, we are concerned with the computation of the subleading term of the polynomials. The main result, Theorem~\ref{th:subleading}, shows that this is a positive number both for $A_\pi(t)$; it is thus a special case of Conjecture~\ref{conj:posX}. We give two proofs of this result, from which we derive some nice formulas mixing hook lengths and contents of partitions (Proposition \ref{prop:newhookformulas}). Section~\ref{sec:firstroot} is concerned with the proof that if $\{1,2n\}$ is not an arch in $\pi$, then  $A_\pi(-1)=0$; this is a special case of Conjecture~\ref{conj:realroots}. The proof relies on the multivariate polynomial extension of $\Psi_\pi$, the main properties of which are recalled briefly. 

 Section~\ref{sec:tau} deals with certain bivariate polynomials $\Psi_\pi(\tau,t)$ which specialize to $A_\pi(t)$ when $\tau=1$. It turns out that the conjectures of Section~\ref{sec:conj} generalize in a very satisfying way. We finally give two appendices: Appendix~\ref{app:equivab} gives a proof of the technical result in Theorem~\ref{th:equivab}, while Appendix~\ref{app:examples} lists some data on the polynomials $A_\pi(t)$.


\section{Definitions}
\label{sec:defi}

We first introduce matchings and different notions related to them. We then describe Fully Packed Loop configurations, as well as the Completely Packed Loop model.

\subsection{Matchings}
\label{representations}
 A matching\footnote{our matchings are usually called {\em perfect noncrossing matchings} in the literature, but this is the only kind of matchings we will encounter so there will be no possible confusion.} $\pi$ of size $n$ is defined as a set of $n$ disjoint pairs of integers $\{1,\ldots,2n\}$, which are {\em noncrossing} in the sense that if $\{i,j\}$ and $\{k,l\}$ are two pairs in $\pi$ with $i<j$ and $k<l$, then it is forbidden to have $i<k<j<l$ or $k<i<l<j$. We will represent matchings by sets of arches on $2n$ horizontally aligned points labeled from $1$ to $2n$. There are $\frac{1}{n+1}\binom{2n}{n}$ matchings with $n$ pairs, which is the famous $n$th Catalan number.
Matchings can be represented by other equivalent objects:

\begin{itemize}
 \item A well-formed sequence of parentheses, also called \emph{parenthesis word}. Given an arch in a matching, the point connected to the left (respectively to the right) is encoded by an opening parenthesis (resp. by a closing parenthesis);
\[
 \begin{tikzpicture}[scale=0.25]
  \draw[arche] (0,0) .. controls (0,.5) and (1,.5) .. (1,0);
  \draw[arche] (2,0) .. controls (2,1.5) and (5,1.5) .. (5,0); 
  \draw[arche] (3,0) .. controls (3,.5) and (4,.5) .. (4,0);
  \draw[line] (-.5,0) -- (5.5,0);
 \end{tikzpicture}
\Leftrightarrow ()(())
\]

 \item A Dyck Path, which is a path between $(0,0)$ and $(2n,0)$ with steps NE $(1,1)$ and SE $(1,-1)$ that never goes under the horizontal line $y=0$. An opening parenthesis corresponds to a NE step, and a closing one to a SE step;
\[
 ()(()) \Leftrightarrow
 \begin{tikzpicture}[scale=0.25, baseline=2pt]
  \draw[dyck] (0,0) -- (1,1) -- (2,0) -- (3,1) -- (4,2) -- (5,1) -- (6,0);
 \end{tikzpicture}
\]

 \item A Young diagram is a collection of boxes, arranged in left-justified rows, such that the size of the rows is weakly decreasing from top to bottom. Matchings with $n$ arches are in bijection with Young diagrams such that the $i$th row from the top has no more than $n-i$ boxes. The Young diagram can be constructed as the complement of a Dyck path, rotated $45^\circ$ counterclockwise;
\[
 \begin{tikzpicture}[scale=0.25, baseline=3pt]
   \draw[dyck] (0,0) -- (1,1) -- (2,0) -- (3,1) -- (4,2) -- (5,1) -- (6,0);
   \draw[young, dotted] (1,1) -- (3,3);
   \draw[young, dotted] (2,0) -- (4,2);
   \draw[young, dotted] (1,1) -- (2,0);
   \draw[young, dotted] (2,2) -- (3,1);
   \draw[young, dotted] (3,3) -- (4,2);
 \end{tikzpicture}
\Leftrightarrow
 \begin{tikzpicture}[scale=0.25, baseline=-10pt]
   \draw[young] (0,0) -- (0,-2);
   \draw[young] (1,0) -- (1,-2);
   \draw[young] (0,0) -- (1,0);
   \draw[young] (0,-1) -- (1,-1);
   \draw[young] (0,-2) -- (1,-2);
 \end{tikzpicture}
\]

\item A sequence $a=\{a_1,\ldots,a_n\}\subseteq\{1,\ldots,2n\}$, such that $a_{i-1}<a_i$ and $a_i\leq 2i-1$ for all $i$. Here $a_i$ is the position of the $i$th opening parenthesis.
\[
 ()(()) \Leftrightarrow \{1,3,4\}
\]
\end{itemize}

We will often identify matchings under those different representations, through the bijections explained above. We may need at times to stress a particular representation: thus we write $Y(\pi)$ for the Young diagram associated to $\pi$, and $a(\pi)$ for the increasing sequence associated to $\pi$, etc...

We will represent $p$ nested arches around a matching $\pi$  by ``$(\pi)_p$'', and $p$ consecutive small arches by ``$()^p$''; thus for instance 
\[
((((()()))))()()()=(()^2)_4()^3.
\]

We define a {\em partial order} on matchings as follows: $\si \leq \pi$ if the Young diagram of $\pi$ contains the Young diagram of $\si$, that is $Y(\si)\subseteq Y(\pi)$. In the Dyck path representation, this means that the path corresponding to $\si$ is always weakly above the path corresponding to $\pi$; in the sequence representation, if we write $a=a(\si)$ and $a'=a(\pi)$, then this is simply expressed by $a_i\leq a'_i$ for all $i$.

Given a matching $\pi$, we define $d(\pi)$ as the total number of boxes in the Young diagram $Y(\pi)$. We also let $\pi^*$ be the conjugate matching of $\pi$, defined by: $\{i,j\}$ is an arch in $\pi^*$ if and only if $\{2n+1-j,2n+1-i\}$ is an arch in $\pi$. This corresponds to a mirror symmetry of the parenthesis word, and a transposition in the Young diagram. We also define a natural {\em rotation} $r$ on matchings: $i,j$ are linked by an arch in $r(\pi)$ if and only if $i+1,j+1$ are linked in $\pi$ (where indices are taken modulo $2n$). These last two notions are illustrated on Figure~\ref{fig:matchings}.

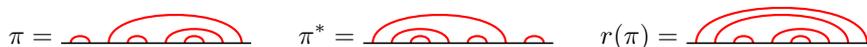
\begin{figure}[!ht]
\begin{align*}
\pi&=
 \begin{tikzpicture}[scale=0.25]
  \draw[arche] (0,0) .. controls (0,.5) and (1,.5) .. (1,0);
  \draw[arche] (2,0) .. controls (2,2) and (9,2) .. (9,0); 
  \draw[arche] (3,0) .. controls (3,.5) and (4,.5) .. (4,0);
  \draw[arche] (5,0) .. controls (5,1) and (8,1) .. (8,0);
  \draw[arche] (6,0) .. controls (6,.5) and (7,.5) .. (7,0);
  \draw[line] (-.5,0) -- (9.5,0);
 \end{tikzpicture}
 &
 \pi^*&=
 \begin{tikzpicture}[scale=0.25]
  \draw[arche] (9,0) .. controls (9,.5) and (8,.5) .. (8,0);
  \draw[arche] (7,0) .. controls (7,2) and (0,2) .. (0,0); 
  \draw[arche] (6,0) .. controls (6,.5) and (5,.5) .. (5,0);
  \draw[arche] (4,0) .. controls (4,1) and (1,1) .. (1,0);
  \draw[arche] (3,0) .. controls (3,.5) and (2,.5) .. (2,0);
  \draw[line] (-.5,0) -- (9.5,0);
 \end{tikzpicture}
 &
 r(\pi)&=
 \begin{tikzpicture}[scale=0.25]
  \draw[arche] (0,0) .. controls (0,2.5) and (9,2.5) .. (9,0);
  \draw[arche] (1,0) .. controls (1,2) and (8,2) .. (8,0); 
  \draw[arche] (2,0) .. controls (2,.5) and (3,.5) .. (3,0);
  \draw[arche] (4,0) .. controls (4,1) and (7,1) .. (7,0);
  \draw[arche] (5,0) .. controls (5,.5) and (6,.5) .. (6,0);
  \draw[line] (-.5,0) -- (9.5,0);
 \end{tikzpicture}
\end{align*}
\caption{A matching, its conjugate, and the rotated matching.\label{fig:matchings}}
\end{figure}

We need additional notions related to the Young diagram representation. So let $Y$ be a young diagram, and $u$ one of its boxes. The {\em hook length} $h(u)$ is the number of boxes below $u$ in the same column, or to its right in the same row (including the box $u$ itself). We note $H_Y$ the product of all hook lengths, i.e. $H_Y=\prod_{u\in Y} h(u)$. The {\em content} $c(u)$ is given by $y-x$ if $u$ is located in the $x$th row from the top and the $y$th column from the left; we write $u=(x,y)$ in this case. The {\em rim} of $Y$ consists of all boxes of $Y$ which are on its southeast boundary; removing the rim of a partition leaves another partition, and repeating this operation until the partition is empty gives us the {\em rim decomposition} of $Y$.

\subsection{Fully Packed Loops} 
\label{sub:FPLintro}

 A {\emph Fully Packed Loop configuration} (FPL) of size $n$ is a subgraph of the square grid with $n^2$ vertices, such that each vertex is connected to exactly two edges. We furthermore impose the following boundary conditions: the grid is assumed to have n external edges on each side, and we select alternatively every second of these edges to be part of our FPLs. By convention, we fix that the topmost external edge on the left boundary is part of the selected edges, which fixes thus the entire boundary of our FPLs. We number these external edges counterclockwise from $1$ to $2n$, see Figure~\ref{fig:fplexample}.
\begin{figure}[!ht]
 \begin{center}
\includegraphics[page=5,width=0.7\textwidth]{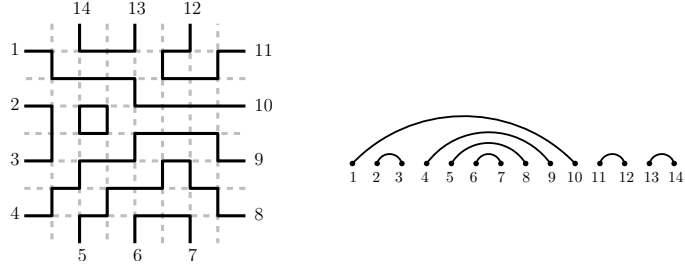}
\end{center}
\caption{FPL with its associated matching \label{fig:fplexample}}
\end{figure}

In each FPL configuration $F$ the chosen external edges are clearly linked by paths which  do not cross each other. We define $\pi(F)$ as the set of pairs $\{i,j\}$ of integers in $\{1,\ldots,2n\}$ such that the external edges labeled $i$ and $j$ are linked by a path in $F$. Then $\pi(F)$ is a matching in the sense of Section~\ref{representations}; an example is given on the right of  Figure~\ref{fig:fplexample}. 

\begin{defi}[$A_\pi$]
 For any matching $\pi$, we define $A_\pi$ as the number of FPLs $F$ such that $\pi(F)=\pi$.
\end{defi}

A result of Wieland~\cite{wieland} shows that a rotation on matchings leaves the numbers $A_\pi$ invariant, and it is then easily seen that conjugation of matchings also leaves them invariant:

\begin{thm}[\cite{wieland}]
\label{thm:invar_api}
For any matching $\pi$, we have $A_\pi=A_{r(\pi)}$ and $A_\pi=A_{\pi^*}$.
\end{thm}

 Now we let $A_n$ be the total number of FPLs of size $n$; by definition we have $A_n=\sum_\pi A_\pi$ where $\pi$ goes through all matchings with $n$ arches. We also define $A_{n}^V$ as the number of FPLs of size $n$ which are invariant with respect to vertical symmetry. It is easily seen that $A_{2n}^V=0$. We have the famous product expressions of these quantities:
\begin{align}
 A_n&=\prod_{k=0}^{n-1} \frac{(3k+1)!}{(n+k)!}; \\
A_{2n+1}^V&= \frac{1}{2^n}\prod_{k=1}^n\frac{(6k-2)!(2k-1)!}{(4k-1)!(4k-2)!}.
\end{align}

The original proofs can be found in~\cite{Zeil-ASM,Kup-ASM} for $A_n$, and~\cite{MR1954236} for $A_{n}^V$.

\subsection{Completely Packed Loop model}
\label{sub:O1}

In this subsection we explain briefly the Completely Packed Loop Model (CPL) with periodic boundary conditions; for more details see~\cite{artic47, hdr, dG-review}. Let $n$ be an integer, and define a {\em state} as a column vector indexed by matchings of size $n$.

Let $e_i$ be the operator on matchings which creates a new arch at $(i,i+1)$, and join the vertices formerly linked to $i$ and $i+1$, as shown in the following examples:
\begin{align*}
 e_3 
\begin{tikzpicture}[scale=0.25, baseline=-3pt]
 \draw[arche] (0,0) .. controls (0,.5) and (1,.5) .. (1,0);
 \draw[arche] (2,0) .. controls (2,1.5) and (5,1.5) .. (5,0); 
 \draw[arche] (3,0) .. controls (3,.5) and (4,.5) .. (4,0);
 \draw[line] (-.5,0) -- (5.5,0);
\end{tikzpicture} = 
\begin{tikzpicture}[scale=0.25, baseline=-3pt]
 \draw[arche] (0,0) .. controls (0,.5) and (1,.5) .. (1,0);
 \draw[arche] (2,0) .. controls (2,1.5) and (5,1.5) .. (5,0); 
 \draw[arche] (3,0) .. controls (3,.5) and (4,.5) .. (4,0);
 \draw[line] (-.5,0) -- (5.5,0);
 \draw[arche] (0,0) -- (0,-1);
 \draw[arche] (1,0) -- (1,-1);
 \draw[arche] (2,0) .. controls (2,-.5) and (3,-.5) .. (3,0); 
 \draw[arche] (2,-1) .. controls (2,-.5) and (3,-.5) .. (3,-1); 
 \draw[arche] (4,0) -- (4,-1);
 \draw[arche] (5,0) -- (5,-1);
 \draw[line] (-.5,-1) -- (5.5,-1);
\end{tikzpicture} &= 
\begin{tikzpicture}[scale=0.25, baseline=-3pt]
 \draw[arche] (0,0) .. controls (0,.5) and (1,.5) .. (1,0);
 \draw[arche] (2,0) .. controls (2,.5) and (3,.5) .. (3,0); 
 \draw[arche] (4,0) .. controls (4,.5) and (5,.5) .. (5,0);
 \draw[line] (-.5,0) -- (5.5,0);
\end{tikzpicture}\\
 e_4
\begin{tikzpicture}[scale=0.25, baseline=-3pt]
 \draw[arche] (0,0) .. controls (0,.5) and (1,.5) .. (1,0);
 \draw[arche] (2,0) .. controls (2,1.5) and (5,1.5) .. (5,0); 
 \draw[arche] (3,0) .. controls (3,.5) and (4,.5) .. (4,0);
 \draw[line] (-.5,0) -- (5.5,0);
\end{tikzpicture} = 
\begin{tikzpicture}[scale=0.25, baseline=-3pt]
 \draw[arche] (0,0) .. controls (0,.5) and (1,.5) .. (1,0);
 \draw[arche] (2,0) .. controls (2,1.5) and (5,1.5) .. (5,0); 
 \draw[arche] (3,0) .. controls (3,.5) and (4,.5) .. (4,0);
 \draw[line] (-.5,0) -- (5.5,0);
 \draw[arche] (0,0) -- (0,-1);
 \draw[arche] (1,0) -- (1,-1);
 \draw[arche] (3,0) .. controls (3,-.5) and (4,-.5) .. (4,0); 
 \draw[arche] (3,-1) .. controls (3,-.5) and (4,-.5) .. (4,-1); 
 \draw[arche] (2,0) -- (2,-1);
 \draw[arche] (5,0) -- (5,-1);
 \draw[line] (-.5,-1) -- (5.5,-1);
\end{tikzpicture} &=
\begin{tikzpicture}[scale=0.25, baseline=-3pt]
 \draw[arche] (0,0) .. controls (0,.5) and (1,.5) .. (1,0);
 \draw[arche] (2,0) .. controls (2,1.5) and (5,1.5) .. (5,0); 
 \draw[arche] (3,0) .. controls (3,.5) and (4,.5) .. (4,0);
 \draw[line] (-.5,0) -- (5.5,0);
\end{tikzpicture}
\end{align*}
The operator $e_0$ creates an arch linking the positions 1 and 2n. Attached to these operators is the {\em Hamiltonian}
\[
 \mathcal{H}_{2n}=\sum_{i=0}^{2n-1} (1-e_i),
\]
where $1$ is the identity.  $\mathcal{H}_{2n}$ acts naturally on states, and the groundstate $(\Psi_\pi)_{\pi:|\pi|=n}$ attached to $\mathcal{H}_{2n}$ is defined as follows:

\begin{defi}[$\Psi_\pi$]
\label{defi:psipi}
Let $n$ be a positive integer. We define the groundstate in the Completely Packed Loop model as the  vector $\Psi=(\Psi_\pi)_{\pi:|\pi|=n}$ which is the solution of $\mathcal{H}_{2n}\Psi=0$, normalized by $\Psi_{()_n}=1$.
\end{defi}

By the Perron-Frobenius theorem, this is well defined. We have then the followings properties:

\begin{thm}
\label{th:propPsipi}
 Let $n$ be a positive integer.
\begin{itemize}
\item For any $\pi$,  $\Psi_{r(\pi)}=\Psi_{\pi^*}=\Psi_{\pi}$.
\item The numbers $\Psi_\pi$ are positive integers.
\item $\sum_\pi \Psi_\pi = A_n$, where the sum is over matchings such that $|\pi|=n$.
\end{itemize}
\end{thm}

The stability by rotation and conjugation is clear from the symmetry of the problem. The integral property was proved in~\cite[Section 4.4]{artic43}, while the sum rule was proved in~\cite{artic31}. The computation of this groundstate has received a lot of interest, mainly because of the Razumov--Stroganov (ex-)conjecture.

\subsection{The Razumov--Stroganov conjecture}

A simple computation shows that
\begin{align*}
 \Psi_{
 \begin{tikzpicture}[scale=0.15]
   \draw[arche] (0,0) .. controls (0,.5) and (1,.5) .. (1,0);
   \draw[arche] (2,0) .. controls (2,.5) and (3,.5) .. (3,0); 
   \draw[arche] (4,0) .. controls (4,.5) and (5,.5) .. (5,0);
 \end{tikzpicture}}
&=2&
 \Psi_{
 \begin{tikzpicture}[scale=0.15]
   \draw[arche] (0,0) .. controls (0,1.5) and (5,1.5) .. (5,0);
   \draw[arche] (1,0) .. controls (1,.5) and (2,.5) .. (2,0); 
   \draw[arche] (3,0) .. controls (3,.5) and (4,.5) .. (4,0);
 \end{tikzpicture}}
&=2 &
 \Psi_{
 \begin{tikzpicture}[scale=0.15]
   \draw[arche] (0,0) .. controls (0,1) and (3,1) .. (3,0);
   \draw[arche] (1,0) .. controls (1,.5) and (2,.5) .. (2,0); 
   \draw[arche] (4,0) .. controls (4,.5) and (5,.5) .. (5,0);
 \end{tikzpicture}}
&=1\\ 
\Psi_{
 \begin{tikzpicture}[scale=0.15]
   \draw[arche] (0,0) .. controls (0,.5) and (1,.5) .. (1,0);
   \draw[arche] (2,0) .. controls (2,1) and (5,1) .. (5,0); 
   \draw[arche] (3,0) .. controls (3,.5) and (4,.5) .. (4,0);
 \end{tikzpicture}}
&=1 &
 \Psi_{
 \begin{tikzpicture}[scale=0.15]
   \draw[arche] (0,0) .. controls (0,1.5) and (5,1.5) .. (5,0);
   \draw[arche] (1,0) .. controls (1,1) and (4,1) .. (4,0); 
   \draw[arche] (2,0) .. controls (2,.5) and (3,.5) .. (3,0);
 \end{tikzpicture}}
&=1
\end{align*}
which are exactly the numbers that appear in the FPL counting:
\medskip

\begin{center}
\includegraphics[page=4,width=0.7\textwidth]{FPL_Tiago_fig}
\end{center}

\medskip
Razumov and Stroganov~\cite{RS-conj} noticed in 2001 that this seems to hold in general, and this was recently proved by Cantini and Sportiello~\cite{ProofRS}:

\begin{thm}[Razumov--Stroganov conjecture]
\label{conj:rs}
 The groundstate components of the Completely Packed Loop model count the number of FPL configurations: for any matching $\pi$,
\[
 \Psi_\pi=A_{\pi}.
\]
\end{thm}

The proof of Cantini and Sportiello consists in verifying that the relations of Definition~\ref{defi:psipi} hold for the numbers $A_\pi$. We note also that the results of Theorem~\ref{th:propPsipi} are now a corollary of the Razumov--Stroganov conjecture.


\section{Matchings with nested arches and polynomials}
\label{sec:polynomials}

\subsection{Definitions and results}

In~\cite{Zuber-conj}, Zuber computed some $\Psi_{(\pi)_p}$ for some small matchings $\pi$, and $p=0,1,2,...$. Among other things, he conjectured the following:
\begin{thm}[{\cite{CKLN,artic47}}]
\label{zuber}
For any matching $\pi$ and $p$ a nonnegative integer, the quantity $A_{(\pi)_p}$ can be written in the following form:
\[
 A_{(\pi)_p}=\frac{P_\pi (p)}{d(\pi)!},
\]
where $P_\pi (p)$ is a polynomial in $p$ of degree $d(\pi)$ with integer coefficients, and leading coefficient equal to $d(\pi)!/H_{\pi}$.
\end{thm}

This was proved first by Caselli, Krattenthaler, Lass and Nadeau in~\cite{CKLN} for $A_{(\pi)_p}$, and by Fonseca and Zinn-Justin in~\cite{artic47} for $\Psi_{(\pi)_p}$. Because of this polynomiality property, we introduce the following notations :

\begin{defi}[$A_\pi(t)$ and $\Psi_\pi(t)$] 
We let $A_\pi(t)$ (respectively  $\Psi_\pi(t)$) be the polynomial in $t$ such that  $A_\pi(p)=A_{(\pi)_p}$ (resp. $\Psi_\pi(p)=\Psi_{(\pi)_p}$) for all integers $p\geq 0$.
\end{defi}

By the Razumov--Stroganov conjecture~\ref{conj:rs} one has clearly for all $\pi$:
\[
 A_\pi(t)=\Psi_\pi(t).
\]
We introduced two different notations so that the origin of the quantities involved becomes clearer; in most of this paper however we will only use the notation $A_\pi(t)$. It is the objective of this paper to investigate these polynomials, and give evidence that they possess very interesting properties, in particular when they are evaluated at negative integers. The following proposition sums up some properties of the polynomials.

\begin{prop}
\label{prop:polynomials}
The polynomial $A_\pi(t)$ has degree $d(\pi)$ and leading coefficient $1/H_\pi$. Furthermore, we have $A_\pi(t)=A_{\pi^*}(t)$, and $A_{(\pi)_\ell}(t)=A_{\pi}(t+\ell)$ for any nonnegative integer $\ell$.
\end{prop}

  The first part comes from Theorem~\ref{zuber}, while the rest is clear when $t$ is a nonnegative integer and thus holds true in general by polynomiality in $t$. 

In this section we will recall briefly certain expressions for these polynomials, and point to other works for the proofs.

\subsection{The FPL case}

If $\pi$ is a matching with $n$ arches, the polynomial $A_\pi(t)$ admits the following expression:
\begin{equation}
\label{eq:apiX}
A_\pi(t)=\sum_{\si\leq \pi}a_{\si}^\pi\cdot S_\si(t-n+1),
\end{equation}
in which $\si$ is a parenthesis word (cf. Section~\ref{representations}), the $a_{\si}^\pi$ are the nonnegative integers denoted by $a(\si,\pi,\mathbf{0}_n)$ in~\cite{Thapper}, and $S_\si(t-n+1)$ is the polynomial given by
\[ 
 S_\si(t-n+1)=\frac{1}{H_\si}\prod_{u\in Y(\si)}(t-n+1+c(u)),
\]
in which and $H_\si$, $c(u)$ being defined in Section~\ref{representations}. If $N$ denotes a nonnegative integer, $S_\si(N)$ enumerates semistandard Young tableaux of shape $Y(\si)$ with entries not larger than $N$: this is the {\em hook content formula}, cf.~\cite{StanleyEnum2} for instance.
\medskip

Equation~\eqref{eq:apiX} above can be derived from~\cite[Equation (4)]{Thapper}  (itself based on the work ~\cite{CKLN}) together with  Conjecture~3.4 in the same paper: this conjecture and the derivation are proved in~\cite{NadFPL1}.

\subsection{The CPL case} \label{sec:pol_qKZ}

In this subsection we briefly explain how to compute bivariate polynomials $\Psi_{\pi}(\tau,t)$, defined as the homogeneous limit of certain multivariate polynomials (see Section~\ref{sec:firstroot} for more details and references). We will be mostly interested in the case $\tau=1$, since we recover the groundstate $\Psi_{\pi}(t)=\Psi_{\pi}(1,t)$, as explained in~\cite{hdr}; we address the case of general $\tau$ in Section~\ref{sec:tau}.

So let $a=\{a_1,\ldots,a_n\}$ be a matching represented as an increasing sequence, and define the polynomial $\Phi_{a}(\tau)$ by:
\[
 \Phi_{a}(\tau) = \oint \ldots \oint \prod_i \frac{du_i}{2 \pi i u_i^{a_i}} \prod_{j>i} (u_j-u_i) (1+\tau u_j+u_i u_j). 
\]
We can then obtain the $\Psi_\pi(\tau)$ via a certain matrix $C(\tau)$ :
\begin{align} 
 \Phi_a (\tau)=&\sum_\pi C_{a,\pi}(\tau) \Psi_\pi (\tau);\label{eq:psiphi1}\\
 \Psi_\pi (\tau)=&\sum_a \C_{\pi,a}(\tau) \Phi_a (\tau).\label{eq:psiphi2}
\end{align}

The coefficients $C_{a,\pi}(\tau)$ are given explicitly in~\cite[Appendix A]{artic41}. We just need the following facts:

\begin{prop}[{\cite[Lemma 3]{artic47}}] 
\label{prop:Bases}
 Let $a$ and $\pi$ be two matchings. Then we have:
\[
 C_{a,\pi}(\tau)=\begin{cases}
             0 & \textrm{if } \pi \nleq a;\\
             1 & \textrm{if } \pi=a;\\
             P_{a,\pi} (\tau) & \textrm{if } \pi < a,
            \end{cases}
\]
where $P_{a,\pi}(\tau)$ is a polynomial in $\tau$ with degree $\leq d(a)-d(\pi)-2$.
\end{prop}

Moreover, we have
\begin{equation}
\label{eq:capi-tau}
 C_{a,\pi}(\tau)=(-1)^{d(a)-d(\pi)} C_{a,\pi}(-\tau),
\end{equation}
since it is a product of polynomials $U_s$ in $\tau$ with degree of the form $d(a)-d(\pi)-2k$, $k\in \mathbb{N}$, and parity given by $d(a)-d(\pi)$: this is an easy consequence of~\cite[p.12 and Appendix C]{artic47}. 

By abuse of notation, we write $(a)_p$ to represent $\{1,\ldots,p,p+a_1,\ldots,p+a_n\}$, since this corresponds indeed to adding $p$ nested arches to $\pi(a)$ via the bijections of Section~\ref{sec:defi}. Then 
one easy but important lemma for us is the following:

\begin{lemma}[{\cite[Lemma 4]{artic47}}]
\label{lem:sta} 
The coefficients $C_{a,\pi}(\tau)$ are stable, that is:
  \[
   C_{(a)_p,(\pi)_p}(\tau)=C_{a,\pi}(\tau) \qquad \forall p\in \mathbb{N}.
 \]
\end{lemma}

We remark that Proposition~\ref{prop:Bases}, Equation~\eqref{eq:capi-tau} and Lemma~\ref{lem:sta} also hold for the coefficients $\C_{a,\pi}(\tau)$ of the inverse matrix. Now
\begin{align*}
 \Phi_{(a)_p} (\tau) =& \oint\ldots\oint \prod_i^{n+p} \frac{du_i}{2\pi i u_i^{\hat{a}_i}} \prod_{j>i} (u_j-u_i)(1+\tau u_j+u_i u_j)\\
 =& \oint\ldots\oint \prod_i^n \frac{du_i}{2\pi i u_i^{a_i}} (1+\tau u_i)^p \prod_{j>i} (u_j-u_i)(1+\tau u_j+u_i u_j),
\end{align*}
where we integrated in the first $p$ variables and renamed the rest $u_{p+i}\mapsto u_i$. This is a polynomial in $p$, and we will naturally note $\Phi_{a} (\tau,t)$ the polynomial such that $\Phi_{a} (\tau,p)=\Phi_{(a)_p}(\tau)$.

Finally, from Equation~\eqref{eq:psiphi2} and Lemma~\ref{lem:sta} we obtain the fundamental equation
\begin{equation}
\label{eq:psitauphitau}
 \Psi_\pi (\tau,t) = \sum_a \C_{\pi,a}(\tau) \Phi_a (\tau,t).
\end{equation}

In the special case $\tau=1$, we write $C_{a,\pi}=C_{a,\pi}(1)$, $\Phi_a (t)=\Phi_a (1,t)$ and thus
\begin{equation}
\label{eq:psipiX}
 A_\pi(t)=\Psi_\pi (t) = \sum_a \C_{\pi,a} \Phi_a (t),
\end{equation}
thanks to the Razumov--Stroganov conjecture~\ref{conj:rs}. This gives us a second expression for $A_\pi (t)$, the first one being given by~\eqref{eq:apiX}.


\section{The main conjectures}
\label{sec:conj}

In this section we present several conjectures about the polynomials $A_{\pi}(t)$. For each of them, we will give strong supporting evidence. We will first give a combinatorial construction that is essential in the statement of the conjectures.

\subsection{Combinatorics}
\label{sub:combdef}
We give two rules which define certain integers attached to a matching $\pi$. It turns out that the two rules are equivalent, which is the content of Theorem~\ref{th:equivab}.

 Let $\pi$ be a link pattern, and $n=|\pi|$ its number of arches. We let $Y(\pi),d(\pi)$ be the Young diagram of $\pi$ and its number of boxes respectively, as defined in Section~\ref{representations}. We also use the notation $\widehat{x}=2n+1-x$ for $x\in \lb 1,2n \rb$.
\medskip

{\bf Rule $A$:} For $p$ between $1$ and $n-1$, we consider the set $\mathcal{A}_p^L(\pi)$ of arches $\{a_1,a_2\}$ such that $a_1\leq p$ and $p<a_2<\widehat{p}$, and the set $\mathcal{A}_p^R(\pi)$ of arches $\{a_1,a_2\}$ such that $p<a_1<\widehat{p}$ and  $\widehat{p}\leq a_2$. It is clear that $|\mathcal{A}_p^L(\pi)|+|\mathcal{A}_p^R(\pi)|$ is an even integer, and we can thus define the integer $\ma_p(\pi)$ by
\[
\ma_p(\pi):=\frac{|\mathcal{A}_p^L(\pi)|+|\mathcal{A}_p^R(\pi)|}{2}.
\] 

For instance, let $\pi_0$ be the matching with $8$ arches represented below on the left; we give an alternative representation on the right by folding the second half of the points above the first half, so that $\widehat{x}$ and $x$ are vertically aligned. For $p=4$, we get $|\mathcal{A}_p^L(\pi_0)|=3,|\mathcal{A}_p^R(\pi_0)|=1$, which count arches between the regions (O) and (I), and thus $\ma_4(\pi_0)=4/2=2$. The reader will check that \[\ma_p(\pi_0)=0,1,2,2,2,1,1\] for $p=1,\ldots,7$.

\begin{center}
\includegraphics[page=3,width=0.9\textwidth]{FPL_Tiago_fig}
\end{center}

{\bf Rule B:} Label the boxes of $Y(\pi)$ by associating $n+1-x-y$  to the box $(x,y)$. Then decompose $Y(\pi)$ in rims (cf. Section~\ref{representations}) and let $R_1,\ldots,R_t$ be the successive rims: using the example $\pi_0$ from rule A, we represented below the $Y(\pi_0)$ with its labeling and decomposition in (three) rims. For a given rim $R_\ell$, denote by $i$ and $j$ the labels appearing at the bottom left and top right of the rim, and by $k$ the minimal value appearing in the rim (so that $k\leq i,j$). We define the multiset $B_\ell$ as
 \[
\{k\}\cup\{i,i-1,\ldots,k+1\}\cup\{j,j-1,\ldots k+1\},
\]
 and let $B_\pi$ be the union of all multisets $B_\ell$. Finally, we define $\mb_i(\pi)$ be the multiplicity of the integer $i\in\{1,\ldots,n-1\}$ in $B_\pi$.

In the case of $\pi_0$, the rims give the multisets $\{2,4,3,3\}$, $\{4,5,5\}$ and $\{6,7\}$. Their union is $B_{\pi_0}=\{2,3^2,4^2,5^2,6,7\}$, so that \[\mb_p(\pi_0)=0,1,2,2,2,1,1\] for $p=1,\ldots,7$.

\begin{center}
\includegraphics[page=2,width=0.2\textwidth]{FPL_Tiago_fig}
\end{center}

We see here that $\ma_p(\pi_0)=\mb_p(\pi_0)$ for all $p$, which holds in general:

\begin{thm}
\label{th:equivab}
 For any matching $\pi$, and any integer $p$ such that $1\leq p\leq |\pi|-1$, we have $\ma_p(\pi)=\mb_p(\pi)$.
\end{thm}

The proof of this theorem is a bit technical, but not difficult; it is given in Appendix~\ref{app:equivab}. 

\begin{defi}[$m_p(\pi)$]
 For any matching $\pi$ and any integer $p$, we let $m_p(\pi)$ be the common value of $\ma_p(\pi)$ and $\mb_p(\pi)$ if $1\leq p\leq |\pi|-1$, and be equal to $0$ otherwise.
\end{defi}

We have then the following result:

\begin{prop}
\label{prop:compatroots}
For any matching $\pi$, we have $\sum_pm_p(\pi)\leq d(\pi)$, and the difference $d(\pi)-\sum_pm_p(\pi)$ is an even integer.
\end{prop}

\begin{proof}
 Rule B is more suited to prove this proposition. We will clearly get the result if we can prove that for each rim $R_t$, the number of boxes $r_t$ in $R_t$ is greater or equal than the cardinality $b_t$ of the multiset $B_t$, and the difference between the two quantities is even. Therefore we fix a rim $R_t$, and we use the notations $i,j,k$ from the definition of Rule B. We compute easily  $r_t=2n-i-j-1$ while $b_t=i+j-2k+1$. The difference is thus $\delta_t:=r_t-b_t=2(k+n-1-(i+j))$, which is obviously even. It is also nonnegative: indeed, if $c, c'$ are the extreme boxes with the labels $i,j$ respectively, then the minimal value of $k$ is obtained if the rim consists of the boxes to the right of $c$ together with the boxes below $c'$. At the intersection of these two sets of boxes, the value of $k$ is equal to $i+j-n+1$, which shows that $\delta_t$ is nonnegative and completes the proof.
\end{proof}

We will use this result in Section~\ref{sub:realroots}.

\subsection{The conjectures}
\label{sub:conjectures}

  The rest of this section will consist of the statement of Conjectures ~\ref{conj:realroots}, ~\ref{conj:dec},~\ref{conj:gpi} and ~\ref{conj:posX}, together with evidence in their support. The first three conjectures are related to values of the polynomials $A_\pi(t)$ when the argument $t$ is a negative integer; what these conjectures imply is that some mysterious combinatorics occur around these values $A_\pi(-p)$. The fourth conjecture states simply that the polynomials $A_\pi(t)$ have positive coefficients, and is thus slightly different in spirit than the other ones, though they are clearly related.

The principal evidence in support of the conjectures, as well as the source of their discovery, is the following result:

\begin{fact}
Conjectures ~\ref{conj:realroots}, ~\ref{conj:dec} and ~\ref{conj:posX} are true for all matchings $\pi$ such that $\pi\leq 8$. Conjecture~\ref{conj:gpi} is true for all $n\leq 8$. 
\end{fact}

The corresponding polynomials $A_\pi(t)$ were indeed computed in Mathematica for these values of $\pi$ thanks to Formula~\ref{eq:capi-tau}, and each conjecture was then checked from these exact expressions; note that there are 1430 matchings $|\pi|$ such that $|\pi|=8$. In Appendix~\ref{app:examples} we list the polynomials $A_\pi(t)$ for $|\pi|=4$. 

\subsubsection{Real roots}
\label{sub:realroots}

The first conjecture gives a complete description of all real roots of the polynomials $A_{\pi}(t)$:

\begin{conj} 
\label{conj:realroots}
 All the real roots of the polynomials $A_{\pi}(t)$ are negative integers, and $-p$ appears with multiplicity $m_p(\pi)$. Equivalently, we have a factorization:
\[
 A_{\pi}(t) = \frac{1}{|d(\pi)|!} \cdot \left(\prod_{p=1}^{|\pi|-1} (t+p)^{m_p(\pi)}\right)\cdot Q_{\pi} (t),
\]
where $Q_{\pi} (t)$ is a polynomial with integer coefficients and no real roots. 
\end{conj}

We must verify first that the definition of the multiplicities is coherent with this conjecture. We know indeed by Theorem~\ref{zuber} that $A_\pi(t)$ has degree $d(\pi)$ in $t$; furthermore  the degree of $Q_{\pi}(t)$ is necessarily even, since it is a real polynomial with no real roots. This means that the sum of the $m_p(\pi)$ should not be larger than $d(\pi)$, and should be of the same parity: this is precisely the content of Proposition~\ref{prop:compatroots}.

 It is also immediately checked that the conjecture is compatible with the two stability properties from Proposition~\ref{prop:polynomials}, that is $A_\pi(t)=A_{\pi^*}(t)$ and $A_{(\pi)_\ell}(t)=A_{\pi}(t+\ell)$ for any nonnegative integer $\ell$. Indeed $m_p(\pi)=m_p(\pi^*)$ is immediately seen from either one of the rules, as is $m_{p+\ell}\left((\pi)_\ell\right)=m_p(\pi)$.
\medskip

As an example, the polynomial for the matching $\pi_0$ of Section~\ref{sub:combdef} is:
\begin{align*}
 A_{\pi_0}(t)=&\frac{(2 + t)(3 + t)^2(4 + t)^2(5 + t)^2(6 + t)(7 + t)}{145152000}\\
   &\times (9 t^6+284 t^5+4355 t^4+39660 t^3+225436 t^2+757456 t+123120) \label{ex_n8},
\end{align*}

In the articles~\cite{artic27} for the FPL case, and~\cite{artic38} for the CPL case, the following formula was established: 
\[
 A_{()_a()_b}(t)=\prod_{i=1}^a \prod_{j=1}^b \frac{t+i+j-1}{i+j-1}.
\]

This is exactly what Conjecture~\ref{conj:realroots} predicts in this case (the constant factor is given by Theorem~\ref{zuber}). This is perhaps easier to see with the definition of the $m_i(\pi)$ by rule B. Here the Young diagram is a rectangle, and it is easily seen that each box will correspond to a root of the polynomial, matching precisely the expression above.

 There is an extension of this ``rectangular'' case in the article~\cite{CasKrat}, the results of which can be reformulated as a computation of the polynomials $A_\pi(t)$ when the diagram $Y(\pi)$ is formed of a rectangle together with one more line consisting of one or two boxes, or two more lines with one box each. Then a simple rewriting of the formulas of Theorems~3.2 and ~4.2 in~\cite{CasKrat} shows that the polynomials have indeed\footnote{we did not actually prove that the polynomials $Q_\pi(t)$ only have complex roots when they are of degree 4, though we tested several values; when $Q_\pi(t)$ has degree 2, then from the explicit form in \cite[Theorem~3.2]{CasKrat} one checks that it has a negative discriminant.} the form predicted by Conjecture~\ref{conj:realroots}. 
\medskip  

In Section~\ref{sec:firstroot}, we will give another piece of evidence for the conjecture, by showing that $-1$ is a root of $A_\pi(t)$ as predicted, that is when there is no arch between $1$ and $2n$ in the matching $\pi$; note though that we will not prove that we have multiplicity $m_1(\pi)=1$ in this case.

\subsubsection{Values for some negative parameters}
We are now interested in the values of the polynomial $A_\pi(t)$ is, when the argument $t$ is specialized to a negative integer which is not a root. Note first that although $A_\pi(t)$ does not have integer coefficients, we have the following:
\begin{prop}
\label{prop:integervalues}
 Let $\pi$ be a matching, $p>0$ an integer; then $A_\pi(-p)$ is an integer.
\end{prop}
\begin{proof}
This is standard: for $d=d(\pi)$, the polynomials $\binom{t+d-i}{d},i=0\ldots d$, form a basis of the space of complex polynomials in $t$ of degree $\leq d$. Since $A_\pi(t)$ has degree $d$, we can write 
\begin{equation}\label{eq:expansion}
A_\pi(t)=\sum_{i=0}^ dc_i \binom{t+d-i}{d}.                                                                                                                                                                                                                                                                       \end{equation}
Now $A_\pi(p)=A_{(\pi)_p}$ is an integer when $p$ is a nonnegative integer. Plugging in successively $t=0,1,2,\ldots,d$ in~\eqref{eq:expansion} shows then that $c_0,c_1,\ldots,c_d$ are in fact all integers, which in turn implies that for negative integers $-p$ we have also that $A_\pi(-p)$ is an integer. 
\end{proof}

So let $\pi$ be a matching, and $p\in\lb 0,|\pi|\rb$ be such that $m_p(\pi)=0$. By Rule A in Section~\ref{sub:combdef}, this means that there are no arches that separate the outer part of $\pi$ consisting of the first $p$ and the last $p$ points (denote it by $\alpha$) from the inner part (denote it by $\beta$), as shown in the picture:
\[
 \pi=
 \begin{tikzpicture}[scale=.25,baseline=0pt]
  \fill [blue!10!white] (4,0) ..  controls (4,4) and (-4,4) .. (-4,0) -- (-2,0) .. controls (-2,2) and (2,2) .. (2,0) -- cycle;
  \draw [green, snake=brace, mirror snake, segment amplitude=1pt] (-4,0) -- (-2,0);
  \draw [green, snake=brace, mirror snake, segment amplitude=1pt] (2,0) -- (4,0);
  \draw [black] (-3,-.5) node {\tiny $p$};
  \draw [black] (3,-.5) node {\tiny $p$};
  \draw [black] (0,0) node {$\beta$};
  \draw [black] (0,2) node {$\alpha$};
 \end{tikzpicture}
\]

Here $\alpha$ and $\beta$ can be naturally considered as matchings in their own right (when properly relabeled), and we introduce the notation $\pi=\alpha \circ \beta$ in this situation. It turns out that the following numbers play a special role in our second conjecture:

\begin{defi}[$G_\pi$] For any matching $\pi$ we define 
\[
G_\pi:=A_{\pi}(-|\pi|).
\]
\end{defi}

By Proposition~\ref{prop:integervalues} above, the $G_\pi$ are actually integers. 
The next conjecture says that these numbers seem to appear naturally when evaluating our polynomials at certain negative integers:

\begin{conj} 
\label{conj:dec}
 Let $\pi$ be a matching and $p$ be an integer between $1$ and $|\pi|-1$ such that $m_p(\pi)=0$, and write $\pi=\alpha \circ \beta$ with $|\alpha|=p$. We have then the following factorization:
 \[
  A_{\pi}(-p)=  G_\alpha A_{\beta}.
 \]
\end{conj}

Here we need to verify a certain sign compatibility with Conjecture~\ref{conj:realroots}, which predicts that $A_{\pi}(-p)$ has sign $(-1)^{M_p}$ where $M_p=\sum_{i\leq p}m_i(\pi)$. Now for this range of $i$ we have obviously $m_i(\pi)=m_i(\alpha)$ by rule A, so that $A_{\pi}(-p)$ has sign $(-1)^{d(\alpha)}$ by Proposition~\ref{prop:compatroots}; but this is then (conjecturally) the sign of $G_\alpha$ (cf. Proposition~\ref{prop:propgpi} below), which is coherent with the signs in Conjecture~\ref{conj:dec}.

\subsubsection{Properties of the $G_\pi$}
\label{sub:gpi}
Conjecture~\ref{conj:dec} shows that the numbers $G_\pi$ seem to play a special role in the values of $A_\pi(t)$ at negative integers. 

\begin{prop}
\label{prop:propgpi}
 For any matching $\pi$, $G_\pi = G_{(\pi)}$ and $G_\pi=G_{\pi^*}$. Moreover, Conjecture~\ref{conj:realroots} implies that $sign(G_\pi)=(-1)^{d(\pi)}$.
\end{prop}
\begin{proof}
 The first two properties are immediately derived from the polynomial identities $A_\pi(t+1)=A_{(\pi)}(t)$ and $A_\pi(t)=A_{\pi^*}(t)$ respectively, given in Proposition~\ref{prop:polynomials}. Then, if all real roots of $A_\pi(t)$ are between $-1$ and $1-|\pi|$ as predicted by Conjecture~\ref{conj:realroots}, the sign of $G_\pi$ must be equal to the sign of $(-1)^{d(\pi)}$, since $A_\pi(t)$ has leading term $t^{d(\pi)}/H_{\pi}$ by Theorem~\ref{zuber}.
\end{proof}
 

 We can compute some special cases, corresponding to $Y(\pi)$ being a rectangle, or a rectangle plus an extra row with just one box:

\begin{prop}
 We have $G_{()_a()_b}=(-1)^{ab}$, while $G_{(()())_{a-2}()_b}=(-1)^{ab+1}(a+1)$. 
\end{prop}

This is easily proved by using the explicit formulas for such $\pi$ which were mentioned in Section~\ref{sub:realroots}. Finally, the most striking features about these numbers are conjectural:

\begin{conj}
\label{conj:gpi}
 For any positive integer $n$, we have
   \begin{align}
   \sum_{\pi : |\pi|=n} |G_\pi|&= A_n\quad\text{and}\quad\sum_{\pi : |\pi|=n} G_\pi = (-1)^{\frac{n(n-1)}{2}}\left(A_n^V\right)^2\label{eq:mysterious}\\
     G_{()^n}&=\begin{cases}
               (-1)^{\frac{n(n-1)}{2}}\left(A_{n+1}^V\right)^2\quad \text{if $n$ is even}; \\
               (-1)^{\frac{n(n-1)}{2}}\left(A_{n}^VA_{n+2}^V\right)\quad\text{if $n$ is odd}.  
              \end{cases}\label{eq:mysterious2}
 \end{align}
\end{conj}

The first equality in~\eqref{eq:mysterious} is particularly interesting: it implies that the unsigned integers $|G_\pi|$, when $\pi$ runs through all matchings of size $n$, sum up to $A_n$, the total number of FPL of size $n$. Of course the $A_\pi$  verify exactly this also, but the properties of $G_\pi$ we have just seen show that the sets of numbers have different behaviors. For instance, the stability property $G_\pi = G_{(\pi)}$ fails for $A_\pi$ obviously, while in general $G_{r(\pi)}\neq G_\pi$. Furthermore, $A_{(()())_{a-2}()_b}=a+b-1$ while $G_{(()())_{a-2}()_b}=(-1)^{ab+1}(a+1)$. 
This raises the problem of finding a partition of FPLs of size $n$ --or any other combinatorial object enumerated by $A_n$-- whose blocks $\{\mathcal{G}_\pi\}_{\pi:|\pi|=n}$ verify $|\mathcal{G}_\pi|=|G_{\pi}|$.
\medskip

\noindent {\em Remark:} In fact, part of the conjecture is a consequence of Conjectures~\ref{conj:realroots} and~\ref{conj:dec}. Indeed, it was proved in~\cite{artic47} that, as polynomials, we have:
\begin{equation}
\label{eq:polsumrule}
 A_{{()^n}}(t)=\sum_{\pi:|\pi|=n}A_\pi(t-1)
\end{equation}
 
 If one evaluates this for $t=1-n$, then two cases occur:
\begin{itemize} 
\item \textit{if $n$ is even}, then we have that $1-n$ is a root of $A_{()^n}(t)$ by Conjecture~\ref{conj:realroots}, and we get from~\eqref{eq:polsumrule} that
\[
 \sum_{\pi : |\pi|=n} G_\pi =0;
\]
\item \textit{if $n$ is odd}, then we are in the conditions of Conjecture~\ref{conj:dec}, which tells us that $A_{()^n}(1-n)=G_{()^{n-1}}A_{()}=G_{()^{n-1}}$, and from~\eqref{eq:polsumrule} we have
\[
 \sum_{\pi : |\pi|=n} G_\pi =G_{()^{n-1}}.
\]

\end{itemize}
 
This then proves that the second equality in~\eqref{eq:mysterious} can be deduced from the first case in~\eqref{eq:mysterious2}.

\subsubsection{Positivity of the coefficients}

Our last conjecture is a bit different from the other three ones, in that it does not deal with values of the polynomials, but their coefficients:
\begin{conj}
\label{conj:posX}
 For any $\pi$, the coefficients of $A_\pi(t)$ are nonnegative.
\end{conj}

It seems in fact to be true that the polynomials $Q_\pi(t)$  --whose existence is predicted by Conjecture~\ref{conj:realroots}-- also only have nonnegative coefficients.
\medskip

By Theorem~\ref{zuber}, we know already that $A_\pi(t)$ is of degree $d(\pi)$ with a positive leading coefficient, so we will be interested in the {\em subleading} coefficient, that is, the coefficient of $t^{d(\pi)-1}$. We managed to compute this coefficient and prove that it is indeed positive: this is Theorem~\ref{th:subleading} in the next section. 


\section{The subleading term of the polynomials}
\label{sec:subleading}

In this section we will prove the following result:

\begin{thm}
\label{th:subleading}
 Given a matching $\pi$ of size $n$, $\pi\neq ()_n$, the coefficient of $t^{d(\pi)-1}$ in $A_\pi(t)$ is positive.
\end{thm}

This is a special case of Conjecture~\ref{conj:posX}. We will give two proofs of this theorem, one starting from the expression~\eqref{eq:apiX}, the other based on the expression~\eqref{eq:psipiX}. As a byproduct of these proofs, we will deduce two formulas concerning products of hook lengths (Proposition~\ref{prop:newhookformulas}).

\subsection{First proof}

We use first the expression of $A_\pi(t)$ given by the sum in Equation~\eqref{eq:apiX}:
\[
 A_\pi(t)=\sum_{\si\leq \pi}a_{\si}^\pi\cdot S_\si(t+1-n)
\]

 We need to gather the terms contributing to the coefficient of $t^{d(\pi)-1}$: they are of two kinds, depending on whether $S_\si(t+1-n)$ has degree $d(\si)$ equal to $d(\pi)$ or $d(\pi)-1$. Since $\si\leq\pi$, the first case occurs only for $\si=\pi$, while the second case occurs when $Y(\si)$ is obtained from the diagram $Y(\pi)$ by removing a {\em corner} from this diagram, i.e. a box of $Y(\pi)$ which has no box below it and no box to its right. We denote by $Cor(\pi)$ the set of corners of $Y(\pi)$, and we get:
\[
 [t^{d(\pi)-1}]A_\pi(t)=\frac{a_{\pi}^{\pi}}{H_\pi}\sum_{u\in Y(\pi)}(1-n+c(u)) + 
\sum_{(x,y)\in Cor(\pi)}  \frac{a_{\pi-(x,y)}^{\pi}}{H_{\pi-(x,y)}}.
\]
 It is proved in~\cite{CKLN} that $a_{\pi}^{\pi}=1$, and in~\cite{NadFPL1} that $a_{\pi-(x,y)}^{\pi}=2n-1-y$ when $(x,y)$ belongs to $Cor(\pi)$. We can then rewrite the previous expression as follows:
\[
 \frac{d(\pi)(1-n)}{H_\pi} +\frac{1}{H_\pi}\sum_{u\in Y(\pi)}c(u)+  \sum_{(x,y)\in Cor(\pi)}\frac{(n-1)}{H_{\pi-(x,y)}}+ 
\sum_{(x,y)\in Cor(\pi)} \frac{(n-y)}{H_{\pi-(x,y)}}.
\]
Now the first and third terms cancel each other because of the {\em hook length formula} (see~\cite{StanleyEnum2} for instance), which is equivalent to 
\[
 \frac{d(\pi)}{H_\pi}=\sum_{(x,y)\in Cor(\pi)}\frac{1}{H_{\pi-(x,y)}}. 
\]

Therefore we are left with

\begin{equation}
\label{eq:FPL_sdt}
 [t^{d(\pi)-1}]A_\pi(t)=\frac{1}{H_\pi}\sum_{u\in Y(\pi)}c(u)+\sum_{(x,y)\in Cor(\pi)} \frac{(n-y)}{H_{\pi-(x,y)}}.
\end{equation}

We now wish to prove that this is positive, which is not clear since the first term can be negative. The idea is to remember that $A_\pi(t)=A_{\pi^*}(t)$ by Proposition~\ref{prop:polynomials}. Now when  $\pi\mapsto\pi^*$, the box $(x,y)$ is sent to $(y,x)$, all contents change signs,  $Cor(\pi)$ is sent to $Cor(\pi^*)$, and hook lengths are preserved. From these observations we get the alternative  expression:
\begin{equation}
\label{eq:FPL_sdt2}
 [t^{d(\pi)-1}]A_\pi(t)=-\frac{1}{H_\pi}\sum_{u\in Y(\pi)}c(u)+\sum_{(x,y)\in Cor(\pi)} \frac{(n-x)}{H_{\pi-(x,y)}}.
\end{equation}

 Clearly in both \eqref{eq:FPL_sdt} and \eqref{eq:FPL_sdt2} the second term is positive, since $y<n$ for all boxes $(x,y)$ in $Y(\pi)$ (there is at least one such box because $\pi \neq ()_n$). Adding~\eqref{eq:FPL_sdt} and~\eqref{eq:FPL_sdt2}, and dividing by $2$, we obtain that the coefficient $[t^{d(\pi)-1}]A_\pi(t)$ is positive:
\begin{equation}
\label{eq:coeff_pos_expression}
 [t^{d(\pi)-1}]A_\pi(t)=\sum_{(x,y)\in Cor(\pi)} \frac{(2n-x-y)}{H_{\pi-(x,y)}}.
\end{equation}

\subsection{Second proof}
\label{sub:subleadingO1}
Here we use the results of Section~\ref{sec:pol_qKZ}, with $\tau=1$. Equation~\eqref{eq:psipiX} says that
\[
\Phi_a(t) = A_\pi(t) + \sum_{\sigma<\pi} C_{\pi,\sigma} A_\sigma (t),
\]
where $a=a(\pi)$. By Theorem~\ref{zuber}, we know that $A_\pi(t)$ has degree $d(\pi)$. Furthermore, since $C_{\pi,\sigma}$ has degree $\leq d(\pi)-d(\sigma) -2$ if $\sigma<\pi$, we conclude that the coefficient of $t^{d(\pi)-1}$ in $A_\pi(t)$ and $\Phi_{a(\pi)}(t)$ is the same, so:
\[
 [t^{d(\pi)-1}]A_{\pi}(t)= [t^{d(\pi)-1}]\oint\ldots\oint \prod_{i=1}^{|a|} \frac{du_i}{2 \pi i u_i^{a_i}}
 (1+u_i)^t \prod_{j>i} (u_j-u_i) (1+u_j+u_i u_j).
\]

If we consider $(1+u_j+u_i u_j)=(1+u_j)+u_i u_j$, we notice that each time we pick the term $u_i u_j$, we decrease $a_i$ and $a_j$ by $1$ and thus the integral corresponds formally to a diagram with two boxes less, so the degree in $t$ decreases by $2$ also; these terms can thus be ignored, which gives:
\begin{align*}
 [t^{d(\pi)-1}]A_{\pi}(t)=& [t^{d(\pi)-1}]\oint \ldots \oint \prod_i \frac{du_i}{2 \pi i u_i^{a_i}} (1+u_i)^{t+i-1} \prod_{j>i} (u_j-u_i)\\
=&[t^{d(\pi)-1}]\sum_{\sigma \in S_{|\pi|}} (-1)^\sigma \oint \ldots \oint \prod_i \frac{du_i}{2 \pi i a_i+1-\sigma_i} (1+u_i)^{t+i-1}\\
=&[t^{d(\pi)-1}]\sum_\sigma (-1)^{\sigma} \prod_i \binom{t+i-1}{a_i-\sigma_i}\\
=&[t^{d(\pi)-1}]\det \left| \binom{t+i-1}{a_i-j} \right|.
\end{align*}

Expanding the binomial up to the second order, we get:
\[
 \binom{t+i-1}{a_i-j} = t^{a_i-j}\frac{1+\frac{(a_i-j)(2i+j-a_i-1)}{t}}{(a_i-j)!}+\text{terms of lower degree}.
\]

If we compute the subleading term of the determinant we get:
\begin{align}
\label{eq:expdet}
 [t^{d(\pi)-1}]A_{\pi}(t)&=[t^{-1}]\det\left| \frac{1+\frac{(a_i-j)(2i+j-a_i-1)}{t}}{(a_i-j)!} \right|\notag\\
 & = \sum_{k=0}^{n-1} \det\left| \frac{1}{(a_i-j)!} \times
  \begin{cases}
    1& \textrm{if }i\neq k\\
    (a_i-j)(2i+j-a_i-1)/2& \textrm{if }i= k
  \end{cases}
 \right|.
\end{align}

We want to show that this expression is equal to the r.h.s. of~\eqref{eq:FPL_sdt2}. First of all, we need to express the quantities involving hooks and contents in terms of the sequence $a$. Notice that the integer $a_i$ is naturally associated to the $(n+1-i)th$ row from the top in $Y(a)$, the length of this row being given by $(a_i-i)$. 

\begin{itemize}
\item It is well known (see for instance~\cite[p.132]{Sagan_symmgroup}) that 
\begin{equation}
\label{eq:hookdet}
  \frac{1}{H_{Y(a)}} = \det \left| \frac{1}{(a_i-j)!} \right|;
\end{equation}
\item The contents in the row indexed by $a_i$ are given by $i-n,i-n+1,\ldots,i-n+(a_i-i-1)$, which sum up to $\frac{1}{2} (a_i-i)(2n-a_i-i+1)$, and therefore we get
\[
 \sum_{u \in Y(a)} c(u) = \sum_{i=1}^n\frac{1}{2} (a_i-i)(2n-a_i-i+1);
\]
  \item Noticing that $a_i\mapsto a_i-1$ removes a box in $(n+1-i)th$ row, we have:
\begin{equation}
\label{eq:pouet}
 \sum_{(x,y)\in Cor(\pi)} \frac{n-x}{H_{\pi-(x,y)}} = \sum_{k=1}^n \det \left| \frac{1}{(a_i-j)!} 
    \begin{cases}
      1 & \textrm{if } i \neq k \\
      (a_i-j)(i-1) & \textrm{if } i=k
    \end{cases} \right|.
\end{equation}
Here we can sum over all $k$, i.e. all rows, because the determinants corresponding to rows without a corner in $Y(a)$ have two equal rows and thus vanish.
\end{itemize}

Looking back at Equation~\eqref{eq:expdet}, we write  
\[(a_i-j)(2i+j-a_i-1)/2=-(a_i-j)(a_i-j-1)/2+(a_i-j)(i-1),\]
and splitting each determinant in two thanks to  linearity in the $k$th row. Then the expression obtained by summing the determinants  correponding to the second term is precisely~\eqref{eq:pouet}; therefore all that remains to prove is the following lemma:
 
\begin{lemma}
\label{lem:endproof}
\begin{multline}\label{eq:o1_det}
\sum_{k=1}^{n} \det\left| \frac{1}{(a_i-j)!} \times
  \begin{cases}
    1& \textrm{if }i\neq k\\
    (a_i-j)(a_i-j-1)& \textrm{if }i= k
  \end{cases}\right|\\
 = \left(\sum_{k=1}^{n}(a_k-k)(a_k-2n+k-1)\right)\times \det \left| \frac{1}{(a_i-j)!} \right|
\end{multline}
\end{lemma}
 
\begin{proof}
We write $(a_k-k)(a_k-2n+k-1)=a_k(a_k-2n-1)+k(2n-k+1)$ and use linearity of the determinant with respect to line (and column) $k$ to write the r.h.s. of~\eqref{eq:o1_det} as
\begin{multline} \label{eq:cu_det}
 \sum_{k=1}^n \det \left|\frac{1}{(a_i-j)!} 
  \begin{cases} 
   1 & \textrm{if } i \neq k\\
   a_i(a_i-2n-1) & \textrm{if } i=k
  \end{cases}  \right|\\
 +
 \sum_{k=1}^n \det \left|\frac{1}{(a_i-j)!} 
  \begin{cases} 
   1 & \textrm{if } j \neq k\\
   j(2n-j+1) & \textrm{if } j=k
  \end{cases}  \right|.
\end{multline}

Now we notice that we have the general identity for any variables $a_{ij},b_{ij}$:
\[
\sum_{k=1}^n \det \left|a_{ij}
  \begin{cases}
   1 & \textrm{if } i \neq k\\
   b_{ij} & \textrm{if } i=k
  \end{cases}  \right|
=\sum_{k=1}^n \det \left|a_{ij}
  \begin{cases}
   1 & \textrm{if } j \neq k\\
   b_{ij} & \textrm{if } j=k
  \end{cases}  \right|.
\]
Indeed, both correspond to the coefficient of $t^{-1}$ in $\det\left|a_{ij}+a_{ij}b_{ij}/t\right|$, which can be expanded using multilinearity according either to rows or to columns. We use this in the first term of~\eqref{eq:cu_det} and in the l.h.s. in the lemma; putting things together, the r.h.s. of~\eqref{eq:o1_det} minus the l.h.s is equal to:

\[
 \sum_{k=1}^n \det \left|\frac{1}{(a_i-j)!} 
  \begin{cases} 
   1 & \textrm{if } j \neq k\\
   2(n-j)(a_i-j) & \textrm{if } j=k
  \end{cases}  \right|.
\]

For all $k<n$ the determinants have two proportional columns ($k$ and $k+1$), while for $k=n$ the $n$th column of the determinant is zero. So all these determinants are zero and therefore so is their sum, which achieves the proof of the lemma. \end{proof}

This completes the second proof of Theorem~\ref{th:subleading}.

\subsection{Application to hook length products}

It turns out that some of the computations made to prove Theorem~\ref{th:subleading} have nice applications to certain {\em hook identities}. If $Y$ is a Young diagram, let $Cor(Y)$ be its corners, and $HD(Y)$ (respectively $VD(Y)$) be the horizontal (resp. vertical) dominos which can be removed from $Y$, defined as two boxes which can be removed in the same row (resp. the same column). Then we have the following identities:
\begin{prop}
\label{prop:newhookformulas}
 For any Young diagram $Y$ we have:
 \[
 \frac{2\sum_{u\in Y}c(u)}{H_Y}=\sum_{(x,y)\in Cor(Y)} \frac{(y-x)}{H_{Y-(x,y)}}
\]
and
\[
 \frac{2\sum_{u\in Y}c(u) }{H_Y}= \sum_{hd\in HD(Y)} \frac{1}{H_{(Y-hd)}}-\sum_{vd\in VD(Y)} \frac{1}{H_{(Y-vd)}}.
\]
\end{prop}

\begin{proof}
 We consider $a$, a sequence such that $Y(a)=Y$. The first formula consists simply in equating the expressions in \eqref{eq:FPL_sdt} and \eqref{eq:FPL_sdt2}. 

We will see that the second formula is a reformulation of Lemma~\ref{lem:endproof}. We already identified $\frac{2}{H_Y}\sum_{u\in Y}c(u)$ as the r.h.s. of the lemma, so we want identify the sums on dominos with the l.h.s. in Lemma~\ref{lem:endproof}. We note first that the $k$th determinant in ~\eqref{eq:o1_det} is of the form \eqref{eq:hookdet} for the sequence $a^{(k)}$ which coincides with $a$ except $a^{(k)}_k=a_k-2$. There are three different cases to consider: firstly, if $a^{(k)}$ has two equal terms, the corresponding determinant vanishes. Then, if $a^{(k)}$ is increasing, we obtain one of the terms in the sum over $HD(Y)$. Finally, for $a^{(k)}$ to have distinct terms when it's not increasing, it is necessary and sufficient that $a_k=a_{k-1}+1$ and $a_{k-2}<a_k-2$. The sequence obtained by switching $a_k-2$ and $a_{k-1}$ is then strictly increasing; if we exchange the rows in the determinant, we will get a negative sign. It is then easy to verify that such sequences are those obtained by removing a vertical domino from $Y$, which achieves the proof.
\end{proof}

As pointed out to the second author by V. F\'eray~\cite{ferayperso}, both formulas can in fact be deduced from the representation theory of the symmetric group, using the properties of Jucys-Murphy elements~\cite{jucys2,murphy}.


\section{The first root}
\label{sec:firstroot}

In this section we will prove the following theorem

 \begin{thm}
\label{th:firstroot}

 For any matching $\pi$ we have
\[
 \Psi_\pi(\tau,-1) = 
  \begin{cases}
   \Psi_{\pi'}(\tau)& \textrm{if }\pi=(\pi');\\
   0 & \textrm{otherwise.}
  \end{cases}
\]
\end{thm}

This is a special case of Conjecture~\ref{conj:realroots} by setting $\tau=1$:

\begin{cor}
 If $m_1(\pi)=1$, then $(t+1)$ divides the polynomial $A_\pi(t)$.
\end{cor}

Indeed $m_1(\pi)=1$ precisely when there is no arch between $1$ and $2n$ in $\pi$ (cf. Rule A in Section~\ref{sub:combdef}), which means that $\pi$ cannot be written as $(\pi')$. For the same reason, Theorem~\ref{th:firstroot} is in general a special case of Conjecture~\ref{conj:taurealroot}. 

To prove this theorem, we use the multiparameter version of the quantities $\Psi_\pi$.

\subsection{Multiparameter setting} \label{sub:multi}

We recall the principal properties of the multiparameter setting as presented in~\cite{artic47,hdr,artic43}. Note that in fact, it is this setting that was used originally to prove the results of Section~\ref{sec:pol_qKZ}; we presented things backwards because this was not needed outside of this section. 

There exist polynomials in $2n$ variables $\Psi_\pi(z_1,\ldots,z_{2n})$ with coefficients in $\mathbb{C}(q)$, indexed by matchings of size $n$, which are defined as solutions of a certain equation~\cite[Formulas 4.2 and 4.3]{hdr} (related to the qKZ equation introduced by Frenkel and Reshetikhin in~\cite{FR-qkz}), which is a generalization of the eigenvector equation defining the $\Psi_\pi$ (cf. Section~\ref{sub:O1}). Here $q$ and $\tau$ are related by $\tau=-q-q^{-1}$, so that $q=\pm e^{2i\pi/3}$ will give $\tau=1$. One can show that these polynomials form a basis of the following vector space $\mathcal{V}_n$:

\begin{defi}[$\mathcal{V}_n$]
 We define $\mathcal{V}_n$ as the vector space of all homogeneous polynomials in $2n$ variables, with total degree $\delta=n(n-1)$ and partial degree $\delta_i=n-1$ in each variable, which obey to the \emph{wheel condition}:
\[
 \left.P(z_1,\ldots,z_{2n})\right|_{z_k=q^2 z_j=q^4 z_i}=0\qquad \forall k>j>i.
\]
\end{defi}

This vector space has dimension $\frac{(2n)!}{n!(n+1)!}$, the number of matchings of size $|\pi|=n$. The polynomials $\Psi_\pi(z_1,\ldots,z_{2n})$ verify the following important lemma:

\begin{lemma}[\cite{artic41}]\label{lem:dual}
Let $q^\epsilon=\{q^{\epsilon_1},\ldots,q^{\epsilon_{2n}}\}$, where $\epsilon_i=\pm 1$ are such that changing $q^{-1}$ in $($ and changing $q$ in $)$ gives a valid parenthesis word $\pi(\epsilon)$. Then 
\[
 \Psi_\pi(q^\epsilon) = \tau^{d(\pi)} \delta_{\pi,\epsilon},
\]
where $\delta_{\pi,\epsilon}=1$ when we have $\pi(\epsilon)=\pi$.
\end{lemma}

Since the $\Psi_\pi(z_1,\ldots,z_{2n})$ form a basis of $\mathcal{V}_n$, the lemma shows that a polynomial in this space is determined by its value on these points $q^\epsilon$. There is a small variation of this lemma, for the cases with a big arch $(1,2n)$, cf.~\cite[Formula 4.15]{hdr}\footnote{We do not use the same normalization as in~\cite{hdr}.}:
\[
 \Psi_{\pi}(q^{-2},q^{\epsilon},q^2)=\left(\frac{q-1}{q-q^{-1}}\right)^{2(n-1)} \tau^{d(\pi)} q^{-(n-1)} \delta_{(\epsilon),\pi}.
\]

\subsection*{Another basis}We now define another set of polynomials $\Phi_a(z_1,\ldots,z_{2n})$ (indexed by the increasing sequences defined in Section~\ref{representations}),  by the integral formula:
\begin{multline}\label{eq:qKZ_var}
 \Phi_a(z_1,\ldots,z_{2n})= c_n
 \prod_{1\le i<j\le 2n} (qz_i-q^{-1}z_j)\\ \times \oint\ldots\oint \prod_{i=1}^n \frac{dw_i}{2\pi i} \frac{\prod_{1\le i<j\le n}(w_j-w_i)(qw_i-q^{-1}w_j)}{\prod_{1\le k\leq a_i}(w_i-z_k)\prod_{a_i<k\le 2n}(qw_i-q^{-1}z_k)},
\end{multline}
where the integral is performed around the $z_i$ but not around $q^{-2}z_i$, and $c_n=(q-q^{-1})^{-n(n-1)}$. In the limit $z_i=1$ for all $i$ we simply obtain the equations for $\Phi_a(\tau)$ given in Section~\ref{sec:pol_qKZ}, by the change of variables $u_i=\frac{w_i-1}{q w_i - q^{-1}}$. In fact, these polynomials actually also live in $\mathcal{V}_n$ and we have
\begin{align*}
 \Phi_a(z_1,\ldots,z_{2n})=\sum_{\pi} C_{a,\pi}(\tau) \Psi_\pi (z_1,\ldots,z_{2n}),
\end{align*}
where the $C_{a,\pi}(\tau)$ are precisely the coefficients that appear in Section~\ref{sec:pol_qKZ}\footnote{In fact, this is the true definition of these coefficients, and the properties listed in Section~\ref{sec:pol_qKZ} are proved from this definition and~\eqref{eq:C_1}.}. Then
\begin{equation} \label{eq:C_1}
 \Phi_a(q^\epsilon) = \tau^{d(\epsilon)} C_{a,\epsilon}(\tau),
\end{equation}
which is an immediate application of Lemma~\ref{lem:dual}. Using the lemma's variation, we also have: 
\begin{equation} \label{eq:C_2}
 \Phi_a(q^{-2},q^\epsilon,q^2) = \tau^{d(\epsilon)} q^{-(n-1)} \left(\frac{q-1}{q-q^{-1}}\right)^{2(n-1)} C_{a,(\epsilon)}.
\end{equation}

\subsection{The proof}
By Lemma~\ref{lem:sta},
\[
 \Psi_{\pi}(-1)=\sum_a \C_{\pi,a} \Phi_{a} (-1). 
\]

We now introduce the following multiple integral, inspired by Formula~\eqref{eq:qKZ_var}:
\begin{multline}
\label{eq:minusone}
 \Phi_a (z_1,\ldots,z_{2n}|-1):=c_n\frac{z_1 (q-q^{-1})}{qz_1-q^{-1}z_{2n}} \prod_{1 \leq i<j\leq2n} (q z_i -q^{-1} z_j) \\
 \times \oint \ldots \oint \prod_i \frac{dw_i}{2i\pi} \frac{\prod_{i<j}(w_j-w_i)(q w_i -q^{-1}w_j)}{\prod_{j\leq a_i} (w_i-z_j)\prod_{j> a_i} (q w_i-q^{-1} z_j)} \prod_i \frac{q w_i - q^{-1} z_{2n}}{q z_1-q^{-1} w_i}.
\end{multline}

The essential property of $\Phi_a (z_1,\ldots,z_{2n}|-1)$ is that if all $z_i=1$, then we get $\Phi_{a}(-1)$; this requires the change of variables $u_i=\frac{w_i-1}{q w_i - q^{-1}}$ already mentioned after Formula~\eqref{eq:qKZ_var}. If we integrate in $w_1$,  we obtain:
\begin{multline*}
 \Phi_a (z_1,\ldots,z_{2n}|-1)=c_n \prod_{i=2}^{2n-1}(qz_i-q^{-1}z_{2n})\prod_{2\leq i<j}^{2n-1} (q z_i -q^{-1} z_j)\\
\times \oint \ldots \oint \prod_{i=2} \frac{dw_i}{2i\pi} \frac{\prod_{i<j}(w_j-w_i)(q w_i -q^{-1}w_j)}{\prod_{2\leq j\leq a_i} (w_i-z_j)\prod_{2n>j> a_i} (q w_i-q^{-1} z_j)}.
\end{multline*}

The r.h.s. is now factorized in one term which depends on $z_1$ and $z_{2n}$, but not on $a$, and one  which does not depend on $z_1$ and $z_{2n}$, and lives in the vector space $\mathcal{V}_{n-1}$ (with parameters $\{z_2,\ldots,z_{2n-1}\}$). Therefore we can write $\Phi_a(z_1,\ldots,z_{2n}|-1)$ as a linear combination of $\Psi_{\pi}(z_2,\ldots,z_{2n-1})$:

\begin{equation} \label{eq:-1_dec}
 \Phi_{a}(z_1,\ldots,z_{2n}|-1)= \frac{\prod_{i=2}^{2n-1}(qz_i-q^{-1}z_{2n})}{(q-q^{-1})^{2(n-1)}}
 \times \sum_\pi \widehat{C}_{a,\pi}\Psi_{\pi} (z_2,\ldots,z_{2n-1}).
\end{equation}

We have then the following essential lemma:
\begin{lemma}
 For any $a,\epsilon$ we have $\widehat{C}_{a,\epsilon}=C_{a,(\epsilon)}$.
\end{lemma}

\begin{proof}
 First we integrate Formula~\eqref{eq:qKZ_var} in $w_1$:
\begin{multline*}
 \Phi_a(z_1,\ldots,z_{2n})=c_n \prod_{i=2}^{2n-1}(q z_i-q^{-1}z_{2n})\prod_{2\leq i<j<2n} (q z_i -q^{-1} z_j) \\
\times \oint \ldots \oint \prod_{i} \frac{dw_i}{2i\pi} \frac{\prod_{i<j}(w_j-w_i)(q w_i -q^{-1}w_j)}{\prod_{j\leq a_i} (w_i-z_j)\prod_{2n>j> a_i} (q w_i-q^{-1} z_j)}\prod_{i=2}^{2n-1}\frac{qz_1-q^{-1}w_i}{q w_i-q^{-1}z_{2n}}.
\end{multline*}
We then make the substitutions $z_1 \mapsto q^{-2}$ and $z_{2n} \mapsto q^2$:
\begin{multline*}
\Phi_a(q^{-2},z_2,\ldots,z_{2n-1},q^2)=c_n (-1)^{n-1}\prod_{i=2}^{2n-1}(z_i-1) \prod_{2\leq i<j<2n} (q z_i -q^{-1} z_j) \\
\times \oint \ldots \oint \prod_{i=2} \frac{dw_i}{2i\pi} \frac{\prod_{i<j}(w_j-w_i)(q w_i -q^{-1}w_j)}{\prod_{2\leq j\leq a_i} (w_i-z_j)\prod_{2n>j> a_i} (q w_i-q^{-1} z_j)}.
\end{multline*}

Comparing with the formula obtained for $\Phi_{a}(z_1,\ldots,z_{2n}|-1)$, we get:
\[
 \Phi_{a}(z_1,\ldots,z_{2n}|-1)= (-1)^{n-1} \prod_{i=2}^{2n-1}\frac{q z_i-q^{-1}z_{2n}}{z_i-1} \Phi_{a}(q^{-2},z_2,\ldots,z_{2n-1},q^2),
\]
which thanks to~\eqref{eq:-1_dec} becomes:
\[
  \sum_\epsilon \widehat{C}_{a,\epsilon} \Psi_\epsilon(z_2,\ldots,z_{2n-1})=\frac{(q-q^{-1})^{2(n-1)}}{\prod_{i=2}^{2n-1}{z_i-1}} (-1)^{n-1} \sum_{\pi} C_{a,\pi} \Psi_\pi(q^{-2},z_2,\ldots,z_{2n-1},q^2).
\]
Now the l.h.s. lives in $\mathcal{V}_{n-1}$, so it is determined by the points $(q^\sigma)$ (cf. Lemma~\ref{lem:dual} and its variation):
\[
 \sum_\epsilon \widehat{C}_{a,\epsilon} \delta_{\epsilon,\sigma} \tau^{d(\epsilon)}= \sum_\pi C_{a,\pi} \delta_{\pi,(\sigma)} \tau^{d(\pi)},
\]
This simplifies to $\widehat{C}_{a,\sigma}\tau^{d((\sigma))}=C_{a,\sigma} \tau^{d(\sigma)}$; since $d(\sigma)=d((\sigma))$, we get the expected result.
\end{proof}

We can now finish the proof of the theorem. In the limit $z_i=1$ for all $i$, Equation~\eqref{eq:-1_dec} becomes
\[
   \Phi_{a}(-1) = \sum_{\pi:|\pi|=n-1} \widehat{C}_{a,\pi} \Psi_\pi.
 \]

Using the lemma, and multiplying by $\C_{\pi,a}$, this becomes:

\begin{align*}
  \sum_a\C_{\pi,a}\Phi_{a}(-1)=&\sum_a\sum_{\epsilon}\C_{\pi,a}C_{a,(\epsilon)}\Psi_\epsilon \\
  \Leftrightarrow\qquad\qquad\Psi_{\pi}(-1)=&\sum_{\epsilon} \delta_{\pi,(\epsilon)} \Psi_\epsilon,
\end{align*}
which achieves the proof.


\section{\texorpdfstring{The $\tau$ case}{The tau case}}
\label{sec:tau}

The bivariate polynomials $\Psi_{\pi}(\tau,t)$ were introduced in Section~\ref{sec:pol_qKZ}. In this section we present conjectures mimicking those of Section~\ref{sec:conj} for these polynomials.

\subsection{Conjectures}
\label{sub:tauconj}

We will give four conjectures, each of them being in fact a natural extension of one of the conjectures of Section~\ref{sec:conj}. All of these conjectures have been verified for all $\Psi_{\pi}(\tau,t)$ with $|\pi|\leq 8$. We begin with roots:

\begin{conj}
 \label{conj:taurealroot}
 Considering $\Psi_{\pi}(\tau,t)$ as a polynomial in $t$ with coefficients in $\mathbb{Q}[\tau]$, the real roots of $\Psi_{\pi}(\tau,t)$  are negative integers $-p$ and with multiplicity given by $m_p(\pi)$:
\[
 \Psi_{\pi}(\tau,t) = \frac{1}{|d(\pi)|!} \times \prod_{i=1}^{|\pi|} (t+i)^{m_i(\pi)} Q_{\pi} (\tau,t),
\]
where $Q_{\pi} (\tau,t)$ is a polynomial in $t$ with no real roots. 
\end{conj}

 For the example $\pi_0$ of Section~\ref{sub:combdef} we compute:
\begin{align*}
\label{ex_n8}
 \Psi_{\pi_0}(\tau,t)=&\frac{(2 + t)(3 + t)^2(4 + t)^2(5 + t)^2(6 + t)(7 + t)}{145152000} \tau^9\\
   &\times (84000 + 440640 \tau^2 + 151440 t \tau^2 + 13200 t^2 \tau^2 + 523680 \tau^4 + 394360 t \tau^4  \\&+ 
 110520 t^2 + \tau^413670 t^3 \tau^4 + 630 t^4 \tau^4 + 182880 \tau^6 + 211656 t \tau^6  \\&+ 
 101716 t^2 \tau^6 + 25990 t^3 \tau^6+ 3725 t^4 \tau^6 + 284 t^5 \tau^6 + 9 t^6 \tau^6). 
\end{align*}

We then have the natural generalization of the factorization conjecture:
\begin{conj}
\label{conj:taudec}
 Let $\pi$ be a matching and $p$ be a integer between $1$ and $|\pi|-1$ such that $m_p(\pi)=0$, so that $\pi=\alpha \circ \beta$ with $|\alpha|=p$; then 
 \[
  \Psi_{\pi}(\tau,-p)= G_\alpha(\tau) \Psi_{\beta}(\tau).
 \]
\end{conj}

Here $G_\pi (\tau)$ is naturally defined by $ G_\pi (\tau):=\Psi_\pi(\tau,-|\pi|)$, while $\Psi_{\pi}(\tau)$ was defined in Section~\ref{sec:pol_qKZ} and is equal to $\Psi_{\pi}(\tau,0)$. The values for $|\pi|=4$ are given in Appendix~\ref{app:examples}. These $G_{\pi}(\tau)$ present several properties:

\begin{conj}\label{conj:dec_tau}
 We have  $G_{\pi}(\tau)=(-1)^{d(\pi)} g_{\pi}(\tau)$  where $g_\pi{\tau}$ is a polynomial with \emph{nonnegative} integer coefficients. Furthermore, we have the sum rule:
 
\[
 \sum_\pi G_\pi (\tau) = \sum_\pi \Psi_\pi (-\tau).
\]

\end{conj}

We will show in Section ~\ref{sub:lead} that the leading term of $g_\pi(\tau)$ is $\tau^{d(\pi)}$; we will actually compute the leading term in $\tau$ of $\Psi_\pi(\tau,p)$ for various integer values of $p$. Another property of these $G_\pi(\tau)$ is that 
\[
 G_\pi(\tau)=(-1)^{d(\pi)} G_\pi (-\tau),
\]
so that they are odd or even polynomials depending on the parity of $\pi$. More generally, one has $\Psi_\pi(\tau,t)=(-1)^{d(\pi)} \Psi_\pi(-\tau,t)$. Indeed, this is obvious for the polynomials
\[
 \Phi_a = \oint \ldots \oint \prod_i \frac{du_i}{u_i^{a_i}}(1+\tau u_i) \prod_{j>i} (u_j-u_i)(1+\tau u_j+u_i u_j),
\]
and as the basis transformation respects this parity, this holds for $\Psi_\pi(\tau,t)$ as well.


Finally, introducing a $\tau$ doesn't change the positivity:

\begin{conj}
 The bivariate polynomial $d(\pi)!P_\pi (\tau,t)$ has nonnegative integer coefficients.
\end{conj}

\subsection{\texorpdfstring{The leading term of $\Psi_\pi(\tau,p)$}{The leading term in tau}}
 \label{sub:lead}

We now consider $\Psi_\pi(\tau,t)$ as a polynomial in $\tau$, first with coefficients in $\mathbb{C}[t]$, and then with rational coefficients under the specializations $t=p$ for $p$ an integer.

 We start by deriving an expression for the leading term in $\tau$ of the polynomial $\Psi_\pi(\tau,t)$. First we consider the leading term in $\tau$ of $\Phi_a(\tau,t)$ for a given sequence $a$. We have 
\[
 \Phi_a (\tau,t)= \oint\ldots\oint \prod_i \frac{du_i}{2\pi i u_i^{a_i}} (1+\tau u_i)^t \prod_{j>i}(u_j - u_i)(1+\tau u_j+u_i u_j),
\]
It is clear that if we replace $(1+ \tau u_i+u_i u_j)$ for $(1+\tau u_i)$ we don't change the leading term (for the same reasons as in Section~\ref{sub:subleadingO1}). Therefore this last expression has the same leading term in $\tau$ as
\begin{align*}
 & \oint \ldots \oint \prod_i \frac{du_i}{2\pi i u_i^{a_i}} (1+\tau u_i)^{t+i-1} \prod_{j>i}(u_j - u_i)\\
 = & \sum_{\sigma \in S_n} (-1)^\sigma \oint \ldots \oint \prod_i \frac{du_i}{2\pi i u_i^{a_i+1-\sigma_i}} (1+\tau u_i)^{t+i-1}\\
 = & \sum_{\sigma \in S_n} (-1)^\sigma \prod_i \tau^{a_i-\sigma_i} \binom{t+i-1}{a_i-\sigma_i}\\
 = & \tau^{d(a)} \det_{n\times n} \left| \binom{t+i-1}{a_i-j} \right |.
\end{align*}

So we know that the degree in $\tau$ of $\Phi_a (\tau,t)$ is $d(a)$. Because of Equation~\eqref{eq:psitauphitau} 
and Proposition~\ref{prop:Bases}, it is clear that the leading term of $\Psi_\pi (\tau,t)$ is the same as $\Phi_{a(\pi)} (\tau,t)$. We have thus proved:

\begin{prop}
As a polynomial in $\tau$, the leading term of $\Psi_\pi(\tau,t)$ is given by $D_\pi (t)\tau^{d(\pi)}$, where for $a=a(\pi)$ we have
\[
D_\pi (t)=\det_{n\times n} \left| \binom{t+i-1}{a_i-j} \right|.
\]
\end{prop}

Now we turn to what happens when $t$ is specialized to an integer $p$; by definition the cases $p=0$ and $p=-|\pi|$ correspond respectively to the polynomials $\Psi_\pi(\tau)$ and $G_\pi(\tau)$. Clearly if $D_\pi (p)\neq 0$ then the leading term  of $\Psi_\pi(\tau,p)$ is $D_\pi (p)\tau^{d(\pi)}$ by the previous proposition, while if $D_\pi (p)=0$ the leading term is necessarily of smaller degree. Our result is the following:

\begin{thm}
\label{thm:taucoefdom}
Let $\pi$ be a matching, and $p$ be an integer; if $p<0$, we also assume that $\pi$ is not of the form $(\rho)_{|p|}$. Then $D_\pi (p)= 0$ if and only if $1-|\pi|\leq p\leq-1$. Furthermore, \begin{itemize} 
\item if $p\geq 0$ then $D_\pi (p)$ counts the number of tableaux of shape $Y(\pi)$ with entries bounded by $p+|\pi|-1$ which are strictly increasing in rows and columns; \item if $p\leq -|\pi|$, then $(-1)^{d(\pi)}D_\pi(p)$ counts the number of tableaux of shape $Y(\pi)$ with entries bounded by $|p|-|\pi|$ which are weakly increasing in rows and columns;
\item if $1-|\pi|\leq p\leq-1$, then\begin{itemize}
\item if $m_{|p|}(\pi)\neq 0$, Conjecture~\ref{conj:taurealroot} implies that $\Psi_\pi(\tau,p)$ is the zero polynomial;
\item if $m_{|p|}(\pi)= 0$ and $\pi=\alpha \circ \beta$ with $|\alpha|=|p|$, Conjecture~\ref{conj:taudec} implies that the leading term of $\Psi_\pi(\tau,p)$ is given by $(-1)^{d(\alpha)}D_\beta(0)\tau^{d(\alpha)+d(\beta)}$.
\end{itemize}
\end{itemize}
\end{thm}

Note that the condition that $\pi$ is not of the form $(\rho)_{|p|}$ is not a restriction, since in such a case $\Psi_\pi(\tau,p)=\Psi_\rho(\tau,0)$.

\begin{proof}
 We study separately the three cases:
\medskip

 \noindent\emph{Case $p\geq 0$.} The determinant $D_\pi (p)$ is here a particular case of ~\cite[Theorem 6.1]{KrattGenPP}, which says that indeed $D_\pi (p)$ counts tableaux of shape $Y(\pi)$ with entries bounded by $(p+|\pi|-1)$ and increasing in both directions.  For example, if $a(\pi)=\{1,2,4,7\}$ and $p=1$ we get
\[
 D_{\{1,2,4,7\}}(1)=\det_{4\times 4} \left| \binom{i}{a_i-j} \right| = 11,
\]

corresponding to the $11$ tableaux:

\begin{align*}
\begin{tikzpicture}[scale=.25,baseline=0pt]
 \draw[black, thick] (0,0) -- (3,0);
 \draw[black, thick] (0,-1) -- (3,-1);
 \draw[black, thick] (0,-2) -- (1,-2);
 \draw[black, thick] (0,0) -- (0,-2);
 \draw[black, thick] (1,0) -- (1,-2);
 \draw[black, thick] (2,0) -- (2,-1);
 \draw[black, thick] (3,0) -- (3,-1);
 \draw [red] (0.5,-.5) node {\tiny 2};
 \draw [red] (1.5,-.5) node {\tiny 3};
 \draw [red] (2.5,-.5) node {\tiny 4};
 \draw [red] (0.5,-1.5) node {\tiny 4};
\end{tikzpicture}
&&
\begin{tikzpicture}[scale=.25,baseline=0pt]
 \draw[black, thick] (0,0) -- (3,0);
 \draw[black, thick] (0,-1) -- (3,-1);
 \draw[black, thick] (0,-2) -- (1,-2);
 \draw[black, thick] (0,0) -- (0,-2);
 \draw[black, thick] (1,0) -- (1,-2);
 \draw[black, thick] (2,0) -- (2,-1);
 \draw[black, thick] (3,0) -- (3,-1);
 \draw [red] (0.5,-.5) node {\tiny 2};
 \draw [red] (1.5,-.5) node {\tiny 3};
 \draw [red] (2.5,-.5) node {\tiny 4};
 \draw [red] (0.5,-1.5) node {\tiny 3};
\end{tikzpicture}
&&
\begin{tikzpicture}[scale=.25,baseline=0pt]
 \draw[black, thick] (0,0) -- (3,0);
 \draw[black, thick] (0,-1) -- (3,-1);
 \draw[black, thick] (0,-2) -- (1,-2);
 \draw[black, thick] (0,0) -- (0,-2);
 \draw[black, thick] (1,0) -- (1,-2);
 \draw[black, thick] (2,0) -- (2,-1);
 \draw[black, thick] (3,0) -- (3,-1);
 \draw [red] (0.5,-.5) node {\tiny 1};
 \draw [red] (1.5,-.5) node {\tiny 3};
 \draw [red] (2.5,-.5) node {\tiny 4};
 \draw [red] (0.5,-1.5) node {\tiny 4};
\end{tikzpicture}
&&
\begin{tikzpicture}[scale=.25,baseline=0pt]
 \draw[black, thick] (0,0) -- (3,0);
 \draw[black, thick] (0,-1) -- (3,-1);
 \draw[black, thick] (0,-2) -- (1,-2);
 \draw[black, thick] (0,0) -- (0,-2);
 \draw[black, thick] (1,0) -- (1,-2);
 \draw[black, thick] (2,0) -- (2,-1);
 \draw[black, thick] (3,0) -- (3,-1);
 \draw [red] (0.5,-.5) node {\tiny 1};
 \draw [red] (1.5,-.5) node {\tiny 3};
 \draw [red] (2.5,-.5) node {\tiny 4};
 \draw [red] (0.5,-1.5) node {\tiny 3};
\end{tikzpicture}
&&
\begin{tikzpicture}[scale=.25,baseline=0pt]
 \draw[black, thick] (0,0) -- (3,0);
 \draw[black, thick] (0,-1) -- (3,-1);
 \draw[black, thick] (0,-2) -- (1,-2);
 \draw[black, thick] (0,0) -- (0,-2);
 \draw[black, thick] (1,0) -- (1,-2);
 \draw[black, thick] (2,0) -- (2,-1);
 \draw[black, thick] (3,0) -- (3,-1);
 \draw [red] (0.5,-.5) node {\tiny 1};
 \draw [red] (1.5,-.5) node {\tiny 3};
 \draw [red] (2.5,-.5) node {\tiny 4};
 \draw [red] (0.5,-1.5) node {\tiny 2};
\end{tikzpicture}
&&
\begin{tikzpicture}[scale=.25,baseline=0pt]
 \draw[black, thick] (0,0) -- (3,0);
 \draw[black, thick] (0,-1) -- (3,-1);
 \draw[black, thick] (0,-2) -- (1,-2);
 \draw[black, thick] (0,0) -- (0,-2);
 \draw[black, thick] (1,0) -- (1,-2);
 \draw[black, thick] (2,0) -- (2,-1);
 \draw[black, thick] (3,0) -- (3,-1);
 \draw [red] (0.5,-.5) node {\tiny 1};
 \draw [red] (1.5,-.5) node {\tiny 2};
 \draw [red] (2.5,-.5) node {\tiny 4};
 \draw [red] (0.5,-1.5) node {\tiny 4};
\end{tikzpicture}
\\
\begin{tikzpicture}[scale=.25,baseline=0pt]
 \draw[black, thick] (0,0) -- (3,0);
 \draw[black, thick] (0,-1) -- (3,-1);
 \draw[black, thick] (0,-2) -- (1,-2);
 \draw[black, thick] (0,0) -- (0,-2);
 \draw[black, thick] (1,0) -- (1,-2);
 \draw[black, thick] (2,0) -- (2,-1);
 \draw[black, thick] (3,0) -- (3,-1);
 \draw [red] (0.5,-.5) node {\tiny 1};
 \draw [red] (1.5,-.5) node {\tiny 2};
 \draw [red] (2.5,-.5) node {\tiny 4};
 \draw [red] (0.5,-1.5) node {\tiny 3};
\end{tikzpicture}
&&
\begin{tikzpicture}[scale=.25,baseline=0pt]
 \draw[black, thick] (0,0) -- (3,0);
 \draw[black, thick] (0,-1) -- (3,-1);
 \draw[black, thick] (0,-2) -- (1,-2);
 \draw[black, thick] (0,0) -- (0,-2);
 \draw[black, thick] (1,0) -- (1,-2);
 \draw[black, thick] (2,0) -- (2,-1);
 \draw[black, thick] (3,0) -- (3,-1);
 \draw [red] (0.5,-.5) node {\tiny 1};
 \draw [red] (1.5,-.5) node {\tiny 2};
 \draw [red] (2.5,-.5) node {\tiny 4};
 \draw [red] (0.5,-1.5) node {\tiny 2};
\end{tikzpicture}
&&
\begin{tikzpicture}[scale=.25,baseline=0pt]
 \draw[black, thick] (0,0) -- (3,0);
 \draw[black, thick] (0,-1) -- (3,-1);
 \draw[black, thick] (0,-2) -- (1,-2);
 \draw[black, thick] (0,0) -- (0,-2);
 \draw[black, thick] (1,0) -- (1,-2);
 \draw[black, thick] (2,0) -- (2,-1);
 \draw[black, thick] (3,0) -- (3,-1);
 \draw [red] (0.5,-.5) node {\tiny 1};
 \draw [red] (1.5,-.5) node {\tiny 2};
 \draw [red] (2.5,-.5) node {\tiny 3};
 \draw [red] (0.5,-1.5) node {\tiny 4};
\end{tikzpicture}
&&
\begin{tikzpicture}[scale=.25,baseline=0pt]
 \draw[black, thick] (0,0) -- (3,0);
 \draw[black, thick] (0,-1) -- (3,-1);
 \draw[black, thick] (0,-2) -- (1,-2);
 \draw[black, thick] (0,0) -- (0,-2);
 \draw[black, thick] (1,0) -- (1,-2);
 \draw[black, thick] (2,0) -- (2,-1);
 \draw[black, thick] (3,0) -- (3,-1);
 \draw [red] (0.5,-.5) node {\tiny 1};
 \draw [red] (1.5,-.5) node {\tiny 2};
 \draw [red] (2.5,-.5) node {\tiny 3};
 \draw [red] (0.5,-1.5) node {\tiny 3};
\end{tikzpicture}
&&
\begin{tikzpicture}[scale=.25,baseline=0pt]
 \draw[black, thick] (0,0) -- (3,0);
 \draw[black, thick] (0,-1) -- (3,-1);
 \draw[black, thick] (0,-2) -- (1,-2);
 \draw[black, thick] (0,0) -- (0,-2);
 \draw[black, thick] (1,0) -- (1,-2);
 \draw[black, thick] (2,0) -- (2,-1);
 \draw[black, thick] (3,0) -- (3,-1);
 \draw [red] (0.5,-.5) node {\tiny 1};
 \draw [red] (1.5,-.5) node {\tiny 2};
 \draw [red] (2.5,-.5) node {\tiny 3};
 \draw [red] (0.5,-1.5) node {\tiny 2};
\end{tikzpicture}
\end{align*}

Note also that the filling of the shape $Y(\pi)$ where the cell $(x,y)$ is labeled by $x+y-1$ is a valid tableau because $x+y\leq n$ holds for every cell, and therefore $D_\pi(p)> 0$ for $p\geq 0$.

\medskip

\noindent\emph{Case $p\leq -|\pi|$.} We use first the transformation $\binom{N}{k}=(-1)^k\binom{N+k-1}{k}$ for each coefficient in $D_\pi(p)$ to get:
\[
 D_{\pi}(p) = (-1)^{d(\pi)} \det_{n \times n} \left| \binom{|p|+a_i-i-j}{a_i-j} \right|;
\]

Here the sign comes from $(-1)^{a_i-j}=(-1)^{a_i}(-1)^{-j}$ for the coefficient $(i,j)$, with gives the global sign  $(-1)^{\sum_ia_i-\sum_jj}=(-1)^{d(\pi)}$. We can then use ~\cite[Theorem 6.1]{KrattGenPP} in this case also, which gives us that 
$ (-1)^{d(\pi)} D_{\pi}(p)$ counts tableaux of shape $Y(\pi)$ with entries between $0$ and $|p|-|\pi|$ which are weakly increasing in both directions. For the same partition $a(\pi)=\{1,2,4,7\}$ and $p=-5$ we get
\[
 |D_{\pi} (-5)| = \det_{4 \times 4} \left| \binom{5+a_i-i-j}{5-i}\right| = 7,
\]

which corresponds to the $7$ tableaux
\begin{align*}
\begin{tikzpicture}[scale=.25,baseline=0pt]
 \draw[black, thick] (0,0) -- (3,0);
 \draw[black, thick] (0,-1) -- (3,-1);
 \draw[black, thick] (0,-2) -- (1,-2);
 \draw[black, thick] (0,0) -- (0,-2);
 \draw[black, thick] (1,0) -- (1,-2);
 \draw[black, thick] (2,0) -- (2,-1);
 \draw[black, thick] (3,0) -- (3,-1);
 \draw [red] (0.5,-.5) node {\tiny 0};
 \draw [red] (1.5,-.5) node {\tiny 0};
 \draw [red] (2.5,-.5) node {\tiny 0};
 \draw [red] (0.5,-1.5) node {\tiny 0};
\end{tikzpicture}
&&
\begin{tikzpicture}[scale=.25,baseline=0pt]
 \draw[black, thick] (0,0) -- (3,0);
 \draw[black, thick] (0,-1) -- (3,-1);
 \draw[black, thick] (0,-2) -- (1,-2);
 \draw[black, thick] (0,0) -- (0,-2);
 \draw[black, thick] (1,0) -- (1,-2);
 \draw[black, thick] (2,0) -- (2,-1);
 \draw[black, thick] (3,0) -- (3,-1);
 \draw [red] (0.5,-.5) node {\tiny 0};
 \draw [red] (1.5,-.5) node {\tiny 0};
 \draw [red] (2.5,-.5) node {\tiny 1};
 \draw [red] (0.5,-1.5) node {\tiny 0};
\end{tikzpicture}
&&
\begin{tikzpicture}[scale=.25,baseline=0pt]
 \draw[black, thick] (0,0) -- (3,0);
 \draw[black, thick] (0,-1) -- (3,-1);
 \draw[black, thick] (0,-2) -- (1,-2);
 \draw[black, thick] (0,0) -- (0,-2);
 \draw[black, thick] (1,0) -- (1,-2);
 \draw[black, thick] (2,0) -- (2,-1);
 \draw[black, thick] (3,0) -- (3,-1);
 \draw [red] (0.5,-.5) node {\tiny 0};
 \draw [red] (1.5,-.5) node {\tiny 1};
 \draw [red] (2.5,-.5) node {\tiny 1};
 \draw [red] (0.5,-1.5) node {\tiny 0};
\end{tikzpicture}
&&
\begin{tikzpicture}[scale=.25,baseline=0pt]
 \draw[black, thick] (0,0) -- (3,0);
 \draw[black, thick] (0,-1) -- (3,-1);
 \draw[black, thick] (0,-2) -- (1,-2);
 \draw[black, thick] (0,0) -- (0,-2);
 \draw[black, thick] (1,0) -- (1,-2);
 \draw[black, thick] (2,0) -- (2,-1);
 \draw[black, thick] (3,0) -- (3,-1);
 \draw [red] (0.5,-.5) node {\tiny 0};
 \draw [red] (1.5,-.5) node {\tiny 0};
 \draw [red] (2.5,-.5) node {\tiny 0};
 \draw [red] (0.5,-1.5) node {\tiny 1};
\end{tikzpicture}
&&
\begin{tikzpicture}[scale=.25,baseline=0pt]
 \draw[black, thick] (0,0) -- (3,0);
 \draw[black, thick] (0,-1) -- (3,-1);
 \draw[black, thick] (0,-2) -- (1,-2);
 \draw[black, thick] (0,0) -- (0,-2);
 \draw[black, thick] (1,0) -- (1,-2);
 \draw[black, thick] (2,0) -- (2,-1);
 \draw[black, thick] (3,0) -- (3,-1);
 \draw [red] (0.5,-.5) node {\tiny 0};
 \draw [red] (1.5,-.5) node {\tiny 0};
 \draw [red] (2.5,-.5) node {\tiny 1};
 \draw [red] (0.5,-1.5) node {\tiny 1};
\end{tikzpicture}
&&
\begin{tikzpicture}[scale=.25,baseline=0pt]
 \draw[black, thick] (0,0) -- (3,0);
 \draw[black, thick] (0,-1) -- (3,-1);
 \draw[black, thick] (0,-2) -- (1,-2);
 \draw[black, thick] (0,0) -- (0,-2);
 \draw[black, thick] (1,0) -- (1,-2);
 \draw[black, thick] (2,0) -- (2,-1);
 \draw[black, thick] (3,0) -- (3,-1);
 \draw [red] (0.5,-.5) node {\tiny 0};
 \draw [red] (1.5,-.5) node {\tiny 1};
 \draw [red] (2.5,-.5) node {\tiny 1};
 \draw [red] (0.5,-1.5) node {\tiny 1};
\end{tikzpicture}
&&
\begin{tikzpicture}[scale=.25,baseline=0pt]
 \draw[black, thick] (0,0) -- (3,0);
 \draw[black, thick] (0,-1) -- (3,-1);
 \draw[black, thick] (0,-2) -- (1,-2);
 \draw[black, thick] (0,0) -- (0,-2);
 \draw[black, thick] (1,0) -- (1,-2);
 \draw[black, thick] (2,0) -- (2,-1);
 \draw[black, thick] (3,0) -- (3,-1);
 \draw [red] (0.5,-.5) node {\tiny 1};
 \draw [red] (1.5,-.5) node {\tiny 1};
 \draw [red] (2.5,-.5) node {\tiny 1};
 \draw [red] (0.5,-1.5) node {\tiny 1};
\end{tikzpicture}
\end{align*}
\medskip

Now here also $D_\pi(p)\neq 0$ because the tableau filled zeros is valid. For $p=-|\pi|$, this is the only possible tableau and thus the leading coefficient of $G_\pi(\tau)$ is given by $D_\pi(-|\pi|)=(-1)^{d(\pi)}$.
\medskip

\noindent\emph{Case $-|\pi|<p<0$.} We first want to prove that $D_{\pi}(p)=0$ if $\pi$ is not of the form $(\rho)_{|p|}$. We easily check that $\binom{p+i-1}{a_i-j}$ is zero unless either $(i,j)<(|p|+1,a_{|p|+1})$ or $(i,j)\geq (|p|+1,a_{|p|+1})$. Therefore we get a matrix which splits into two rectangular submatrices; the determinant is zero unless these submatrices are square, which means that $|p|+1=a_{p+1}$, and then

\begin{align*}
 D_{\pi}(p) =&
 \det_{|p|\times |p|} \left| \binom{p+i-1}{i-j} \right| \times \det_{(|\pi|-|p|)\times (|\pi|-|p|)} \left| \binom{i-1}{\hat{a}_i-j} \right|\\
 =& D_{\{1,\ldots,-p\}}(p) \times D_{\hat{a}} (0),
\end{align*}
where $\hat{a}_i=a_{r+i}-r$. The first factor is $1$, and the second is non-zero if and only if $\hat{a}$ corresponds to a matching; this is excluded because $\pi$ would be of the form $(\rho)_{|p|}$, which is excluded. Therefore  $D_{\pi}(p)=0$ as wanted.

Now Conjecture~\ref{conj:taurealroot} immediately implies that if $m_{|p|}(\pi)\neq 0$, then $t=p$ is a root of $\Psi_\pi(\tau,t)$, so that $\Psi_\pi(\tau,p)\equiv 0$. If $m_{|p|}(\pi)= 0$, then by Conjecture~\ref{conj:taudec}, the leading term of $\Psi_\pi(\tau,p)$ is equal to the product of the leading terms of $G_\alpha(\tau)$ and $\Psi_\beta(\tau)$. The first one is given by $(-1)^{d(\alpha)}\tau^{d(\alpha)}$ as proved above, while the leading term of $\Psi _\beta(\tau)=\Psi _\beta(\tau,0)$ is given by $D_\beta(0)\tau^{d(\beta)}$, which achieves the proof. 
\end{proof}


\section{Further questions}

\subsection{Solving the conjectures}
Since our paper is centered around conjectures, the most immediate problem is to solve them. We listed four conjectures in Section~\ref{sec:conj} which concern roots, specializations and coefficients of the polynomials $A_\pi(t)$. The difficulty here is that existing expressions for the polynomials $A_\pi(t)$, namely~\eqref{eq:apiX} and~\eqref{eq:psipiX}, consist of certain sums of polynomials, so that it makes it uneasy to find real roots of $A_\pi(t)$, and more generally the sign variations when $t$ is a real variable. For the same reasons, it is hard to figure out where the factorization from Conjecture~\ref{conj:dec} comes from. Furthermore, both expressions~\eqref{eq:apiX} and~\eqref{eq:psipiX} involve negative signs, so that the positivity of coefficients is not at all obvious. One way to attack the conjectures would be then to find new expressions for the polynomials; this could be done by either understanding better the quantities involved in~\eqref{eq:apiX} and~\eqref{eq:psipiX}, or coming up with a new decomposition of the FPLs counted by $A_{(\pi)_p}$ for instance.

Note also that the linear relations from Definition~\ref{defi:psipi}, which determine the $A_\pi$ by the Razumov Stroganov correspondence~\ref{conj:rs}, do not seem to be helpful in the case of nested arches. Indeed given a matching $(\pi)_p$, then the linear relation corresponding to $A_{(\pi)_p}$ involves not only quantities of the form $A_{(\pi')_p}$ or $A_{(\pi')_{p-1}}$, but also $A_{()(\pi)_{p-2}()}=A_{()()(\pi)_{p-2}}$, which is not of the form considered in this work. For two matchings $\pi,\pi'$, the quantities $A_{\pi'(\pi)_p}$ are polynomials in $p$ when $p$ is big enough (cf.~\cite[Theorem 6.7]{CKLN}), and these ones are ``stable'' with respect to these Razumov--Stroganov linear relations: it would be very interesting to study these more general polynomials and find out how our conjectures can be extended.

Another angle to attack some of the conjectures (namely Conj.~\ref{conj:realroots},~\ref{conj:dec} and their $\tau$ counterparts~\ref{conj:taurealroot},~\ref{conj:dec_tau}) would be to extend the approach used in the proof of Theorem~\ref{th:firstroot}: one first needs to extend the multivariate integral definition~\eqref{eq:minusone} to any integer $p$, which can easily be done. The problem is that the expressions obtained are fairly more complicated and intricate than in the case $p=-1$. This is work in progress.

\subsection{Combinatorial reciprocity}
The idea underlying our conjectures (Conjecture~\ref{conj:posX} excepted) is that there should be a ``combinatorial reciprocity theorem'' (\cite{StanleyReciprocity}) attached to these polynomials. That is, we believe there exist yet-to-be-discovered combinatorial objects depending on $\pi$ such that $A_\pi(-p)$ is equal (up to sign) to the number of these objects with size $p$. The most well-known example in the literature of such a phenomenon concerns the {\em Ehrhart polynomial} $i_P(t)$ of a lattice polytope $P$, which counts the number of lattice points in $tP$ when $t$ is a positive integer: for such $t$,  Ehrhart reciprocity then tells us that $(-1)^{\dim P} i_P(-t)$ counts lattice points strictly in $tP$ (see~\cite{BeckRobins} for instance).

It is natural to wonder if our problem fits in the domain of Ehrhart theory, since most known examples of combinatorial reciprocity can be formulated in terms of Ehrhart polynomials: see for instance certain polynomials attached to graphs~\cite{BeckZasnowherezero, BreuerSanyal}. It cannot be a straightforward application however, in the sense $A_\pi(t)$ is not equal to an Ehrhart polynomial $i_P(t)$ in general: indeed, for any lattice polytope $P$ there cannot be two positive integers $i,j$ such that $i_P(-i)i_P(-j)<0$ since such values are either $0$ or of the sign $(-1)^{\dim P}$ by Ehrhart reciprocity. But for $\pi=()()()()=()^4$ for instance, one computes from the explicit expression given in Appendix~\ref{app:examples} that $A_\pi(-2)=-1$ while $A_\pi(-4)=9$. Moreover, one can also show that if Conjecture~\ref{conj:realroots} holds, then given any finite set $S$ of negative integers (included say in $\{-1,\ldots,1-n\}$) there exists a matching $\pi$ of size $n$ such that the set of negative roots of $A_\pi(t)$ is precisely $S$. This is  clearly a behaviour contrasting with Ehrhart polynomials, and even their generalizations to {\em inside-out polytopes} \cite{BeckZasInsideout}.

Conjectures~\ref{conj:realroots} and ~\ref{conj:dec} tell us in particular for which values of $p$ objects counted by $|A_\pi(-p)|$ should exist, and moreover that such objects should {\em split} for certain values of $p$. As pointed out in Section~\ref{sub:gpi}, Conjectures~\ref{conj:dec} and ~\ref{conj:gpi} make it particularly important to figure out what the numbers $G_\pi=A_\pi(-|\pi|)$ count.

\subsection{Consequences of the conjectures}
The conjectures have interesting consequences regarding the numbers  $a_\si^\pi$ involved  in Equation~\eqref{eq:apiX}, since for instance Conjecture~\ref{conj:realroots} directly implies certain linear relations among these numbers. Discovering what these numbers $a_\si^\pi$ are is a step in the direction of a new proof of the Razumov--Stroganov conjecture, in the sense that it gives an expression for $A_\pi$ that could be compared to the expressions for $\Psi_\pi$. We note also that a conjectural expression for these numbers $a_\si^\pi$ was given in ~\cite{ZJtriangle}, which if true would in fact give another proof of the Razumov--Stroganov conjecture; a special case of this expression is proven in~\cite{NadFPL2}.

\appendix


\section{Equivalence of the definitions of root multiplicities}
\label{app:equivab}

We will give here a proof of Theorem~\ref{th:equivab}, which states the integers $\ma_i(\pi)$ and $\mb_i(\pi)$ defined in Section~\ref{sub:combdef} are equal for any matching $\pi$ and any integer $i$ with $1\leq i \leq |\pi|-1$. 

Let $\pi$ be a matching with $n$ arches; we will prove the theorem by induction on $d(\pi)$. The theorem holds if $d(\pi)=0$; indeed this means that $\pi=()_n$, and clearly that $\ma_i(\pi)=\mb_i(\pi)=0$ for all $i$ in this case.
\medskip

We now assume $d(\pi)>0$. Let $\pi'$ be the matching obtained when the external rim of $\pi$ is removed. If $\pi$ is represented as a parenthesis word, then $\pi'$ is simply obtained by replacing the leftmost closing parenthesis of $\pi$ by an opening parenthesis, and the rightmost opening parenthesis by a closing one. Let $i,j,k$ be the indices defined in Rule B. Then in the parenthesis word representing $\pi$, the indices of the two parentheses above are respectively $i+1$ and $\widehat{\jmath}-1$. More precisely, $\pi$ admits the unique factorization:
\begin{equation}
\label{eq:decomp}
\pi=\left(^i)x_1)x_2)\cdots x_{i-k})w(y_{j-k}(\cdots (y_2(y_1(\right)^j,
\end{equation}
where $x_t$, $y_t$ and $w$ are (possibly empty) parenthesis words. We let $a_0:=i+1<a_1<\ldots<a_{i-k}$ be the indices of the closing parentheses written above and $b_{j-k}<\ldots<b_1<b_0=\widehat{\jmath+1}$ be the indices of opening ones.

 Then by the factorization~\eqref{eq:decomp} the matching $\pi$ includes the arches:  
\begin{equation}
\label{eq:archespi}
(k,a_{i-k}),\ldots,(i-1,a_1),(i,i+1)\quad\text{and}\quad (b_{j-k},\widehat{k}),\ldots,(b_1,\widehat{\jmath-1}),(\widehat{\jmath+1},\widehat{\jmath}),
\end{equation}
and moreover these are exactly the arches which are modified when going from $\pi$ to $\pi'$; indeed, these are replaced in $\pi'$ by \begin{align*}
&(k,\widehat{k}),(k+1,\widehat{k+1}),\\
&(k+2,a_{i-k}),\ldots,(i,a_2),(i+1,a_1),\\
&(b_{j-k},\widehat{k-2}),\ldots,(b_2,\widehat{\jmath}),(b_1,\widehat{\jmath+1}).
\end{align*}

\begin{figure}[!ht]

\begin{center}
\includegraphics[page=1,width=0.8\textwidth]{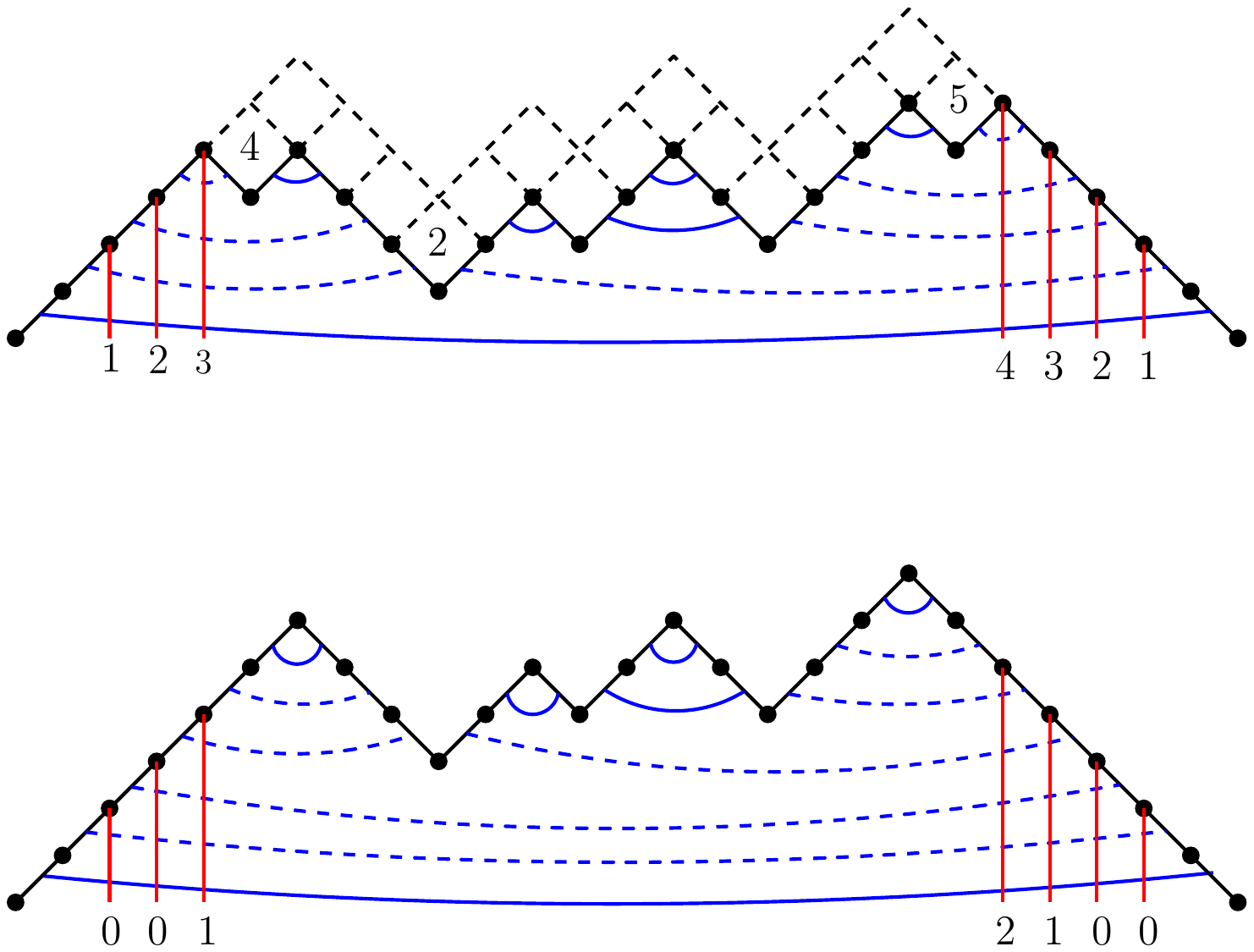}
\end{center}
\caption{In this example we have $i=4,j=5$ and $k=2$, therefore the multiset attached to rim by Rule B is $\{2,3^2,4^2,5\}$.
\label{fig:ruleAB}
}
\end{figure}

From this data we can now study the changes going from $\mathcal{A}_t^L(\pi),\mathcal{A}_t^R(\pi)$ to $\mathcal{A}_t^L(\pi'),\mathcal{A}_t^R(\pi')$ for any integer $t$ between $1$ and $n-1$. A case-by-case analysis shows that:
\[
 |\mathcal{A}_t^L(\pi)|=|\mathcal{A}_t^L(\pi')|+\delta_t,\text{ with } \delta_t=\begin{cases}
1\text{ if } t=k;\\
 2 \text{ if } k<t\leq i;\\
0 \text{ otherwise },                                                                             
\end{cases}
\]

and, symmetrically: 
\[
 |\mathcal{A}_t^R(\pi)|=|\mathcal{A}_t^R(\pi')|+\epsilon_t,\text{ with } \epsilon_t=\begin{cases}
1\text{ if } t=k;\\
2 \text{ if } k<t\leq j;\\
0 \text{ otherwise }.                                                                             
\end{cases}
\]

 By definition  $\ma_t(\pi)-\ma_t(\pi')=(\epsilon_t+\delta_t)/2$. From the explicit values above, this can be equivalently expressed by the fact that the multiset difference between $\{1^{\ma_1(\pi)}2^{\ma_2(\pi)}\ldots\}$ and $\{1^{\ma_1(\pi')}2^{\ma_2(\pi')}\ldots\}$ is:
 \[
  \{k,i,i-1,\ldots,k+1,j,j-1,\ldots k+1\}.
 \]

But this is exactly the multiset associated to the rim of $\pi$ in Rule B, so Theorem~\ref{th:equivab} is proved by induction.


\section{Examples}
\label{app:examples}

We computed the $A_{\pi}(t)$ for all matchings $\pi$ such that $|\pi|\leq 8$. Here is a list of all polynomials for $|\pi|=4$; note that if $\pi\neq\pi^*$ we listed just one of the two since the two polynomials are equal (cf. Proposition~\ref{prop:polynomials}).
\begin{align*}
 A_{\begin{tikzpicture}[scale=0.1]
       \draw[arche] (0,0) .. controls (0,3.5) and (7,3.5) .. (7,0);
       \draw[arche] (1,0) .. controls (1,2.5) and (6,2.5) .. (6,0);
       \draw[arche] (2,0) .. controls (2,1.5) and (5,1.5) .. (5,0);
       \draw[arche] (3,0) .. controls (3,0.5) and (4,0.5) .. (4,0);
    \end{tikzpicture}}
 (t)&=1 \\
 A_{\begin{tikzpicture}[scale=0.1]
       \draw[arche] (0,0) .. controls (0,3.5) and (7,3.5) .. (7,0);
       \draw[arche] (1,0) .. controls (1,2.5) and (6,2.5) .. (6,0);
       \draw[arche] (2,0) .. controls (2,0.5) and (3,0.5) .. (3,0);
       \draw[arche] (4,0) .. controls (4,0.5) and (5,0.5) .. (5,0);
    \end{tikzpicture}}
 (t)&=t+3 \\
 A_{\begin{tikzpicture}[scale=0.1]
       \draw[arche] (0,0) .. controls (0,3.5) and (7,3.5) .. (7,0);
       \draw[arche] (1,0) .. controls (1,1.5) and (4,1.5) .. (4,0);
       \draw[arche] (2,0) .. controls (2,0.5) and (3,0.5) .. (3,0);
       \draw[arche] (5,0) .. controls (5,0.5) and (6,0.5) .. (6,0);
    \end{tikzpicture}}
 (t)&=\frac{1}{2}(t+2)(t+3) \\
 A_{\begin{tikzpicture}[scale=0.1]
       \draw[arche] (0,0) .. controls (0,2.5) and (5,2.5) .. (5,0);
       \draw[arche] (1,0) .. controls (1,1.5) and (4,1.5) .. (4,0);
       \draw[arche] (2,0) .. controls (2,0.5) and (3,0.5) .. (3,0);
       \draw[arche] (6,0) .. controls (6,0.5) and (7,0.5) .. (7,0);
    \end{tikzpicture}}
 (t)&=\frac{1}{6}(t+1)(t+2)(t+3) \\
 A_{\begin{tikzpicture}[scale=0.1]
       \draw[arche] (0,0) .. controls (0,3.5) and (7,3.5) .. (7,0);
       \draw[arche] (1,0) .. controls (1,0.5) and (2,0.5) .. (2,0);
       \draw[arche] (3,0) .. controls (3,0.5) and (4,0.5) .. (4,0);
       \draw[arche] (5,0) .. controls (5,0.5) and (6,0.5) .. (6,0);
    \end{tikzpicture}}
 (t)&=\frac{1}{6}(t+2)(2t^2+11t+21) \\
 A_{\begin{tikzpicture}[scale=0.1]
       \draw[arche] (0,0) .. controls (0,2.5) and (5,2.5) .. (5,0);
       \draw[arche] (1,0) .. controls (1,0.5) and (2,0.5) .. (2,0);
       \draw[arche] (3,0) .. controls (3,0.5) and (4,0.5) .. (4,0);
       \draw[arche] (6,0) .. controls (6,0.5) and (7,0.5) .. (7,0);
    \end{tikzpicture}}
 (t)&=\frac{1}{24}(t+1)(t+2)(3t^2+17t+36) \\
 A_{\begin{tikzpicture}[scale=0.1]
       \draw[arche] (0,0) .. controls (0,1.5) and (3,1.5) .. (3,0);
       \draw[arche] (1,0) .. controls (1,0.5) and (2,0.5) .. (2,0);
       \draw[arche] (4,0) .. controls (4,1.5) and (7,1.5) .. (7,0);
       \draw[arche] (5,0) .. controls (5,0.5) and (6,0.5) .. (6,0);
    \end{tikzpicture}}
 (t)&=\frac{1}{12}(t+1)(t+2)^2(t+3) \\
 A_{\begin{tikzpicture}[scale=0.1]
       \draw[arche] (0,0) .. controls (0,1.5) and (3,1.5) .. (3,0);
       \draw[arche] (1,0) .. controls (1,0.5) and (2,0.5) .. (2,0);
       \draw[arche] (4,0) .. controls (4,0.5) and (5,0.5) .. (5,0);
       \draw[arche] (6,0) .. controls (6,0.5) and (7,0.5) .. (7,0);
    \end{tikzpicture}}
 (t)&=\frac{1}{24}(t+1)(t+2)(t+3)(t^2+4t+12) \\
 A_{\begin{tikzpicture}[scale=0.1]
       \draw[arche] (0,0) .. controls (0,0.5) and (1,0.5) .. (1,0);
       \draw[arche] (2,0) .. controls (2,1.5) and (5,1.5) .. (5,0);
       \draw[arche] (3,0) .. controls (3,0.5) and (4,0.5) .. (4,0);
       \draw[arche] (6,0) .. controls (6,0.5) and (7,0.5) .. (7,0);
    \end{tikzpicture}}
 (t)&=\frac{1}{60}(t+1)(3t^4+27t^3+108t^2+192t+180) \\
 A_{\begin{tikzpicture}[scale=0.1]
       \draw[arche] (0,0) .. controls (0,0.5) and (1,0.5) .. (1,0);
       \draw[arche] (2,0) .. controls (2,0.5) and (3,0.5) .. (3,0);
       \draw[arche] (4,0) .. controls (4,0.5) and (5,0.5) .. (5,0);
       \draw[arche] (6,0) .. controls (6,0.5) and (7,0.5) .. (7,0);
    \end{tikzpicture}}
 (t)&=\frac{1}{180}(t+1)(t+3)(4t^4+32t^3+155t^2+334t+420) \\
\end{align*}

From this list, we can compute the corresponding $G_\pi:=A_\pi(-4)$ (see Section~\ref{sub:gpi}). Here we index them with $Y(\pi)$ instead of $\pi$: this is well defined by the stability property $G_\pi=G_{(\pi)}$. 

\begin{align*}
G_{\begin{tikzpicture}[scale=.1,baseline=0pt]
\draw [black] (0,0) circle (1pt);
\end{tikzpicture}}&=1 &
G_{\begin{tikzpicture}[scale=.1,baseline=0pt]
\draw [black] (0,0) -- (1,0);
\draw [black] (0,0) -- (0,-1);
\draw [black] (1,0) -- (1,-1);
\draw [black] (0,-1) -- (1,-1);
\end{tikzpicture}}&=-1 &
G_{\begin{tikzpicture}[scale=.1,baseline=0pt]
\draw [black] (0,0) -- (2,0);
\draw [black] (0,0) -- (0,-1);
\draw [black] (1,0) -- (1,-1);
\draw [black] (0,-1) -- (2,-1);
\draw [black] (2,0) -- (2,-1);
\end{tikzpicture}}&=1 &
G_{\begin{tikzpicture}[scale=.1,baseline=0pt]
\draw [black] (0,0) -- (2,0);
\draw [black] (0,0) -- (0,-2);
\draw [black] (1,0) -- (1,-2);
\draw [black] (0,-1) -- (2,-1);
\draw [black] (2,0) -- (2,-1);
\draw [black] (0,-2) -- (1,-2);
\end{tikzpicture}}&=-3\\
G_{\begin{tikzpicture}[scale=.1,baseline=0pt]
\draw [black] (0,0) -- (3,0);
\draw [black] (0,-1) -- (3,-1);
\draw [black] (0,0) -- (0,-1);
\draw [black] (1,0) -- (1,-1);
\draw [black] (2,0) -- (2,-1);
\draw [black] (3,0) -- (3,-1);
\end{tikzpicture}}&=-1 &
G_{\begin{tikzpicture}[scale=.1,baseline=0pt]
\draw [black] (0,0) -- (2,0);
\draw [black] (0,-1) -- (2,-1);
\draw [black] (0,-2) -- (2,-2);
\draw [black] (0,0) -- (0,-2);
\draw [black] (1,0) -- (1,-2);
\draw [black] (2,0) -- (2,-2);
\end{tikzpicture}}&=1 &
G_{\begin{tikzpicture}[scale=.1,baseline=0pt]
\draw [black] (0,0) -- (3,0);
\draw [black] (0,0) -- (0,-2);
\draw [black] (1,0) -- (1,-2);
\draw [black] (0,-1) -- (3,-1);
\draw [black] (2,0) -- (2,-1);
\draw [black] (3,0) -- (3,-1);
\draw [black] (0,-2) -- (1,-2);
\end{tikzpicture}}&=4 &
G_{\begin{tikzpicture}[scale=.1,baseline=0pt]
\draw [black] (0,0) -- (3,0);
\draw [black] (0,-1) -- (3,-1);
\draw [black] (0,-2) -- (1,-2);
\draw [black] (0,-3) -- (1,-3);
\draw [black] (0,0) -- (0,-3);
\draw [black] (1,0) -- (1,-3);
\draw [black] (2,0) -- (2,-1);
\draw [black] (3,0) -- (3,-1);
\end{tikzpicture}}&=-9\\
G_{\begin{tikzpicture}[scale=.1,baseline=0pt]
\draw [black] (0,0) -- (3,0);
\draw [black] (0,-1) -- (3,-1);
\draw [black] (0,-2) -- (2,-2);
\draw [black] (0,0) -- (0,-2);
\draw [black] (1,0) -- (1,-2);
\draw [black] (2,0) -- (2,-2);
\draw [black] (3,0) -- (3,-1);
\end{tikzpicture}}&=-3 &
G_{\begin{tikzpicture}[scale=.1,baseline=0pt]
\draw [black] (0,0) -- (3,0);
\draw [black] (0,-1) -- (3,-1);
\draw [black] (0,-2) -- (2,-2);
\draw [black] (0,-3) -- (1,-3);
\draw [black] (0,0) -- (0,-3);
\draw [black] (1,0) -- (1,-3);
\draw [black] (2,0) -- (2,-2);
\draw [black] (3,0) -- (3,-1);
\end{tikzpicture}}&=9
\end{align*}

Finally, here are the $G_\pi(\tau)$ for $|\pi|=4$, as defined in Section~\ref{sec:tau}:

\begin{align*}
G_{\begin{tikzpicture}[scale=.1,baseline=0pt]
\draw [black] (0,0) circle (1pt);
\end{tikzpicture}}&=1 &
G_{\begin{tikzpicture}[scale=.1,baseline=0pt]
\draw [black] (0,0) -- (1,0);
\draw [black] (0,0) -- (0,-1);
\draw [black] (1,0) -- (1,-1);
\draw [black] (0,-1) -- (1,-1);
\end{tikzpicture}}&=-\tau &
G_{\begin{tikzpicture}[scale=.1,baseline=0pt]
\draw [black] (0,0) -- (2,0);
\draw [black] (0,0) -- (0,-1);
\draw [black] (1,0) -- (1,-1);
\draw [black] (0,-1) -- (2,-1);
\draw [black] (2,0) -- (2,-1);
\end{tikzpicture}}&=\tau^2 &
G_{\begin{tikzpicture}[scale=.1,baseline=0pt]
\draw [black] (0,0) -- (0,2);
\draw [black] (0,0) -- (-1,0);
\draw [black] (0,1) -- (-1,1);
\draw [black] (-1,0) -- (-1,2);
\draw [black] (0,2) -- (-1,2);
\end{tikzpicture}}&=\tau^2 \\
G_{\begin{tikzpicture}[scale=.1,baseline=0pt]
\draw [black] (0,0) -- (2,0);
\draw [black] (0,0) -- (0,-2);
\draw [black] (1,0) -- (1,-2);
\draw [black] (0,-1) -- (2,-1);
\draw [black] (2,0) -- (2,-1);
\draw [black] (0,-2) -- (1,-2);
\end{tikzpicture}}&=-2\tau-\tau^3&
G_{\begin{tikzpicture}[scale=.1,baseline=0pt]
\draw [black] (0,0) -- (3,0);
\draw [black] (0,-1) -- (3,-1);
\draw [black] (0,0) -- (0,-1);
\draw [black] (1,0) -- (1,-1);
\draw [black] (2,0) -- (2,-1);
\draw [black] (3,0) -- (3,-1);
\end{tikzpicture}}&=-\tau^3 &
G_{\begin{tikzpicture}[scale=.1,baseline=0pt]
\draw [black] (0,0) -- (1,0);
\draw [black] (0,-1) -- (1,-1);
\draw [black] (0,-2) -- (1,-2);
\draw [black] (0,-3) -- (1,-3);
\draw [black] (0,0) -- (0,-3);
\draw [black] (1,0) -- (1,-3);
\end{tikzpicture}}&=-\tau^3 &
G_{\begin{tikzpicture}[scale=.1,baseline=0pt]
\draw [black] (0,0) -- (2,0);
\draw [black] (0,-1) -- (2,-1);
\draw [black] (0,-2) -- (2,-2);
\draw [black] (0,0) -- (0,-2);
\draw [black] (1,0) -- (1,-2);
\draw [black] (2,0) -- (2,-2);
\end{tikzpicture}}&=\tau^4 \\ 
G_{\begin{tikzpicture}[scale=.1,baseline=0pt]
\draw [black] (0,0) -- (3,0);
\draw [black] (0,0) -- (0,-2);
\draw [black] (1,0) -- (1,-2);
\draw [black] (0,-1) -- (3,-1);
\draw [black] (2,0) -- (2,-1);
\draw [black] (3,0) -- (3,-1);
\draw [black] (0,-2) -- (1,-2);
\end{tikzpicture}}&=3\tau^2+\tau^4 &
G_{\begin{tikzpicture}[scale=.1,baseline=0pt]
\draw [black] (0,0) -- (2,0);
\draw [black] (0,0) -- (0,-3);
\draw [black] (1,0) -- (1,-3);
\draw [black] (0,-1) -- (2,-1);
\draw [black] (2,0) -- (2,-1);
\draw [black] (0,-3) -- (1,-3);
\draw [black] (0,-2) -- (1,-2);
\end{tikzpicture}}&=3\tau^2+\tau^4 &
G_{\begin{tikzpicture}[scale=.1,baseline=0pt]
\draw [black] (0,0) -- (3,0);
\draw [black] (0,-1) -- (3,-1);
\draw [black] (0,-2) -- (1,-2);
\draw [black] (0,-3) -- (1,-3);
\draw [black] (0,0) -- (0,-3);
\draw [black] (1,0) -- (1,-3);
\draw [black] (2,0) -- (2,-1);
\draw [black] (3,0) -- (3,-1);
\end{tikzpicture}}&=-3\tau-5\tau^3-\tau^5&
G_{\begin{tikzpicture}[scale=.1,baseline=0pt]
\draw [black] (0,0) -- (3,0);
\draw [black] (0,-1) -- (3,-1);
\draw [black] (0,-2) -- (2,-2);
\draw [black] (0,0) -- (0,-2);
\draw [black] (1,0) -- (1,-2);
\draw [black] (2,0) -- (2,-2);
\draw [black] (3,0) -- (3,-1);
\end{tikzpicture}}&=-2\tau^3-\tau^5 \\
G_{\begin{tikzpicture}[scale=.1,baseline=0pt]
\draw [black] (0,0) -- (2,0);
\draw [black] (0,-1) -- (2,-1);
\draw [black] (0,-2) -- (2,-2);
\draw [black] (0,0) -- (0,-3);
\draw [black] (1,0) -- (1,-3);
\draw [black] (2,0) -- (2,-2);
\draw [black] (0,-3) -- (1,-3);
\end{tikzpicture}}&=-2\tau^3-\tau^5 &
G_{\begin{tikzpicture}[scale=.1,baseline=0pt]
\draw [black] (0,0) -- (3,0);
\draw [black] (0,-1) -- (3,-1);
\draw [black] (0,-2) -- (2,-2);
\draw [black] (0,-3) -- (1,-3);
\draw [black] (0,0) -- (0,-3);
\draw [black] (1,0) -- (1,-3);
\draw [black] (2,0) -- (2,-2);
\draw [black] (3,0) -- (3,-1);
\end{tikzpicture}}&=3\tau^2+5\tau^4+\tau^6
\end{align*}

\bibliography{DN}
\bibliographystyle{amsplainhyper}
\end{document}